\documentclass{amsart}
\usepackage[english]{babel}
\usepackage{latexsym,amsfonts}
\usepackage{amssymb}
\usepackage{amsmath}
\usepackage{amsthm}

\newtheorem{lemma}{Lemma}[section]
\newtheorem{theorem}[lemma]{Theorem}
\newtheorem{proposition}[lemma]{Proposition}

\theoremstyle{remark}
\newtheorem{rem}[lemma]{Remark}
\newtheorem{ex}[lemma]{Example}

\theoremstyle{definition}
\newtheorem*{notaz}{Notation}

\def\Mat{\mathop{\rm Mat}\nolimits}          

\def\GL{\mathop{\rm GL}\nolimits}  
\def\AGL{\mathop{\rm AGL}\nolimits}           
\def\char{{\rm char\,}}
\def\Ker{{\rm Ker\,}}
\def\blockdiag{\mathop{\rm diag}\nolimits}
\def\diag{\mathop{\rm diag}\nolimits}
\def\rk{\mathrm{rk}}
\def\Imm{\mathop{\rm Im}\,}
\def\Hom{\mathop{\rm Hom}}
\def\sp{\mathrm{Span}}
\def\antidiag{\mathop{\rm antidiag}\nolimits}
\def\s{\sharp}

\def\T{\mathop{\rm T}\nolimits}

\newcommand{\R}{\mathbb{R}}   
\newcommand{\Q}{\mathbb{Q}}  
\newcommand{\F}{\mathbb{F}}    
\newcommand{\A}{\mathcal{L}}    
 \newcommand{\C}{\mathbf{C}}    
 \newcommand{\J}{\mathbf{J}}    
\newcommand{\Tr}{\mathcal{T}}    
\newcommand{\B}{\mathcal{B}}   
 \newcommand{\Z}{\mathbf{Z}}   
\newcommand{\ZZ}{\mathbb{Z}}

\begin{document}

\title{More on regular subgroups of the affine group}

\author{M.A. Pellegrini}
\address{Dipartimento di Matematica e Fisica, Universit\`a Cattolica del Sacro Cuore, Via Musei 41,
I-25121 Brescia, Italy}
\email{marcoantonio.pellegrini@unicatt.it}

\author{M.C. Tamburini Bellani}
\address{Dipartimento di Matematica e Fisica, Universit\`a Cattolica del Sacro Cuore, Via Musei 41,
I-25121 Brescia, Italy}
\email{mariaclara.tamburini@unicatt.it}

 
\begin{abstract}
This paper is a new contribution to the study of regular subgroups of the affine group $\AGL_n(\F)$,
for any field $\F$. In particular we associate to any partition $\lambda\neq (1^{n+1})$ of $n+1$ abelian regular subgroups in such a way that different partitions define non-conjugate subgroups.
Moreover,  we classify the regular subgroups of certain natural types for $n\leq 4$. Our classification is equivalent to the classification of split local algebras of dimension $n+1$ over $\F$.
Our methods, based on classical results of  linear algebra, are computer free.
\end{abstract}

\maketitle

\section{Introduction}

Let $\F$ be any field. 
We identify the affine group $\AGL_n(\F)$ with the subgroup of $\GL_{n+1}(\F)$ consisting of the matrices 
having  $(1,0,\ldots,0)^{\T}$ as first column. With this notation,  $\AGL_n(\F)$ acts on the right on the set
$\mathcal{A}=\left\{(1,v): v\in \F^n\right\}$ of affine points.
Clearly, there exists an epimorphism $\pi: \AGL_n(\F) \rightarrow \GL_n(\F)$ induced by the action
of $\AGL_n(\F)$ on $\F^n$.
A subgroup (or a subset) $R$ of $\AGL_n(\F)$ is called \emph{regular} if it 
acts regularly on $\mathcal{A}$, namely  if, for every $v\in \F^n$, there exists a unique
element in  $R$ having $(1,v)$ as first row.
Thus  $R$  is regular precisely when 
$\AGL_n(\F)=\widehat{\GL}_n(\F)\thinspace R$, with $\widehat{\GL}_n(\F)\cap R =\left\{I_{n+1}\right\}$, 
where $\widehat{\GL}_n(\F)$ denotes the stabilizer  of $(1,0,\ldots,0)$.

A subgroup $H$ of $\AGL_n(\F)$ is \emph{indecomposable} if there exists no decomposition of $\F^n$ as a direct sum of non-trivial $\pi(H)$-invariant subspaces.
Clearly, to investigate 
the structure of regular subgroups, the indecomposable ones are the most relevant, since the other ones are 
direct products of regular subgroups  in smaller dimensions. 
So, one has to expect very many regular subgroups when $n$ is big. 
Actually, in Section \ref{Sec:standard} we show how to construct at least one abelian regular subgroup,
called \emph{standard}, for each partition $\lambda\ne (1^{n+1})$ of $n+1$,
in such a way that different partitions produce non-conjugate subgroups. Several of them are indecomposable.

The structure and the number of  
conjugacy classes of regular subgroups depend on $\F$. For instance, if $\F$ has characteristic $p>0$, every regular subgroup is unipotent 
\cite[Theorem 3.2]{T}, i.e., all its elements satisfy $(t-I_{n+1})^{n+1}=0$. A unipotent group is conjugate to a 
subgroup
of the group of upper unitriangular matrices (see \cite{H}): in particular it has
a non-trivial center. By contrast to the case $p>0$,  
$\AGL_2(\R)$ contains $2^{|\R|}$ conjugacy classes of regular  subgroups with trivial center,
hence not unipotent (see  Example \ref{degio}). 
So, clearly, a classification
in full generality is not realistic. 

Since the center $Z(R)$ of a regular subgroup $R$ is unipotent  
(see Theorem \ref{lin->unip}(a)) if $Z(R)$  is  non-trivial
one may assume, up to conjugation,  that $R$ is contained  in the centralizer 
of a unipotent Jordan form (see Theorem \ref{Jordan form}). But even this condition is weak. 

Before introducing a stronger hypothesis, which allows to treat significant cases, we need some notation. We write every element $r$ of $R$ as
\begin{equation}\label{tauv}
r=\begin{pmatrix} 1 & v \\ 0 & \pi(r)  \end{pmatrix}
=\begin{pmatrix} 1 & v \\ 0 & \tau_R(v) \end{pmatrix}
=\begin{pmatrix} 1 & v \\ 0 & I_n+\delta_R(v)  \end{pmatrix}
=\mu_R(v),  
\end{equation}
where $\mu_R: \F^n\rightarrow \AGL_n(\F)$,  $\tau_R:\F^n\rightarrow \GL_n(\F)$
and $\delta_R:=\tau_R-\textrm{id}: \F^n\rightarrow \Mat_n(\F)$.

The hypothesis we introduce is that $\delta_R$ is linear.
First of all, if $R$ is abelian, then  $\delta_R$ is linear (see \cite{CDS}). 
Moreover, if $\delta_R$ is linear, then $R$ is unipotent by Theorem \ref{lin->unip}(b), but not necessarily abelian.
One further motivation for  this hypothesis is that  $\delta_R$
is linear if and only if  $\A=\F I_{n+1} +R$ is a split local subalgebra of $\Mat_{n+1}(\F)$.
Moreover, two regular subgroups $R_1$ and $R_2$, with $\delta_{R_i}$ linear,
are conjugate in $\AGL_n(\F)$
if and only if the corresponding  algebras $\A_1$ and $\A_2$ are isomorphic (see Section \ref{sec:alg}).
In particular, there is a bijection between conjugacy classes of abelian regular subgroups of  $\AGL_n(\F)$
and isomorphism classes of abelian split local algebras of dimension $n+1$ over $\F$. This fact was first observed in \cite{CDS}.
It was studied also in connection with other algebraic structures in  \cite{CCS2,CR,Ch,Ch2} and 
in \cite{CCS}, where the classification of nilpotent associative algebras given in \cite{DG} is relevant.

In Section \ref{low} we classify, up to conjugation,  certain types of regular subgroups 
$U$ of $\AGL_n(\F)$, for $n\le 4$, and the corresponding algebras.
More precisely, for $n=1$ the only regular subgroup is the translation subgroup $\Tr$, which is standard.
For $n=2,3$, we assume that $\delta_U$ is linear.
If $n=2$ all subgroups $U$, over any $\F$, are standard (Table \ref{Tab4}).
For $n=3$  the abelian regular subgroups  are described in Table \ref{Tab5}. 
The non abelian ones are determined in Lemmas \ref{cosquare} and  \ref{J2:3}: there are $|\F|$ conjugacy classes when $\char \F=2$, $|\F|+1$ otherwise (see also \cite{DG}). 
If $n=4$ we assume that $U$ is abelian. 
The conjugacy classes, when  $\F$ has no quadratic extensions, are shown in Tables \ref{Tab6} and \ref{Tab7}. 

In particular, by the reasons mentioned above, we obtain an independent classification of the  split local algebras of dimension $n\le 4$ over any field. We obtain also the classification of the commutative split local algebras of 
dimension $5$ over fields with no quadratic extensions. Actually, regular subgroups arise from matrix 
representations of these algebras.  
In the abelian case and using the hypothesis that $\F$ is algebraically closed, the same classification, for $n\le 6$, was obtained by Poonen \cite{P}, via commutative algebra.
Namely he presents the algebras as quotients of the polynomial ring $\F[t_1,\dots ,t_r]$, $r\ge 1$.
The two approaches are equivalent but, comparing our results with those of Poonen, we detected  an inaccuracy\footnote{See \cite{P2} for a revised version of \cite{P}.}. 
Namely for $n=5$ and $\char \F=2$, the two algebras defined, respectively, by $\F[x,y,z]/\langle x^2,y^2,xz, yz, xy+z^2\rangle$
and $\F[x,y,z]/\langle xy,xz,yz, x^2+y^2,x^2+z^2 \rangle$ are isomorphic (see Remark \ref{remPoo}).

Our methods, based on linear algebra, are computer independent.

\section{Some basic properties and  examples of regular subgroups}

\begin{lemma}
A  regular submonoid $R$ of $\AGL_n(\F)$ is a 
subgroup. 
\end{lemma}

\begin{proof}
For any $v\in \F^n$, the first row of $\mu_R\left(-v\tau_R(v)^{-1}\right)\mu_R(v)$
is the same as the first row of $I_{n+1}$. 
From the regularity of $R$ it follows  
$\mu_R\left( -v\tau_R(v)^{-1}\right)\mu_R(v)= I_{n+1}$, whence $\mu_R(v)^{-1}=\mu_R\left(-v\tau_R(v)^{-1}\right)\in R$.
Thus $R$ is a subgroup.
\end{proof}

If  $\delta_R\in\Hom_\ZZ\left(\F^n,\Mat_n(\F)\right)$, i.e. $\delta_R$ is additive, direct calculation  gives that a 
regular subset $R$ of $\AGL_n(\F)$ 
containing $I_{n+1}$  is a submonoid, hence a subgroup, if and only if:
\begin{equation}\label{prodotto}
\delta_R(v\delta_R(w))=\delta_R(v)\delta_R(w),\ \quad  \textrm{ for all }  v,w\in \F^n.
\end{equation}
Also, given $v,w\in \F^n$, 
\begin{equation}\label{vivj}
\mu_R(v)\mu_R(w)=\mu_R(w)\mu_R(v) \quad\textrm{ if and only if }\quad v\delta_R(w)= w\delta_R(v).
 \end{equation}

For the next two theorems  we need a basic result, that we recall below for the reader's convenience.
A proof can be found in any text of  linear algebra.

\begin{lemma}\label{Smith} 
Let $g\in \GL_m(\F)$ have characteristic polynomial $\chi_g(t)=f_1(t)f_2(t)$. 
If $\left(f_1(t),f_2(t)\right)=1$, then $g$ is conjugate to $h=\diag(h_1,h_2)$, with $\chi_{h_i}(t)=f_i(t)$, $i=1,2$.
Also, $\C_{\Mat_m(\F)}(h)$ consists of matrices of the same block-diagonal form.
\end{lemma}

\begin{lemma}\label{shape2} 
Let $z$ be an element of a regular subgroup $R$ of $\AGL_n(\F)$.
If $z$ is not unipotent then, up to conjugation of $R$ under $\AGL_n(\F)$, we may suppose that
\begin{equation}\label{shape3}
z=\begin{pmatrix}
1&w_1&0\\
0&A_1&0\\
0&0&A_2
\end{pmatrix},
\end{equation}
where $A_1$ is unipotent and $A_2$ does not have the eigenvalue $1$.
\end{lemma}

\begin{proof}
Let $z=\begin{pmatrix} 1 & w \\ 0 & z_0 \end{pmatrix}$, where $z_0=I_{n}+\delta_R(w)$.
Since $z_0$  has the eigenvalue $1$ \cite[Lemma 2.2]{T}, 
the characteristic polynomial $\chi_{z_0}(t)$ of $z_0$ factorizes 
in $\F[t]$ as 
$$\chi_{z_0}(t)=(t-1)^{m_1}g(t), \quad g(1)\ne 0, \quad m_1\ge 1, \quad \deg g(t)=m_2\ge 1.$$
By Lemma \ref{Smith},
up to conjugation in $\widehat {\GL}_n(\F)$, we may set $z_0=\diag(A_1,A_2)$ with:
$$\chi_{A_1}(t)=(t-1)^{m_1},\quad \chi_{A_2}(t)=g(t).$$
Write $w=\left(w_1, w_2\right)$ with $w_1\in \F^{m_1}$ and 
$w_2\in \F^{m_2}$. Since $A_2$ does not have the eigenvalue $1$, the 
matrix $A_2-I_{m_2}$ is invertible. Conjugating by the translation
$\begin{pmatrix}
1 &  0 &w_2(A_2-I_{m_2})^{-1}\\
0&I_{m_1}&0\\
0 & 0 & I_{m_2}
\end{pmatrix}$ 
we may assume $w_2=0$, i.e.
$z=\begin{pmatrix}
1&w_1&0\\
0&A_1&0\\
0&0&A_2
\end{pmatrix}$.
\end{proof}

We observe that if $\delta_R\in\Hom_\F\left(\F^n,\Mat_n(\F)\right)$, i.e. $\delta_R$ is linear,
then there is a natural embedding of $R$ into a regular subgroup $\widehat R$  of 
$\AGL_n(\widehat \F)$ for any extension $\widehat \F$ of $\F$. Namely:
$$
\widehat R=\left\{\begin{pmatrix}
1&\hat v\\
0&I_n+\delta_{\widehat R}(\hat v)
\end{pmatrix}: \hat v\in \F^n\otimes_\F \widehat \F \right\} \leq \AGL_n(\widehat \F),$$
where $\delta_{\widehat R} \left(\sum_i \hat\lambda v_i\right)=\sum_i \hat \lambda_i \delta_R (v_i)$,  $\hat \lambda_i \in \widehat \F$. Clearly $\delta_{\widehat R}$ is linear.

We obtain the following consequences.

\begin{theorem}\label{lin->unip}
Let $R$  be regular subgroup of $\AGL_n(\F)$. Then the following holds:
\begin{itemize}
\item[(\rm{a})] the center $\Z(R)$ of  $R$ is unipotent;
\item[(\rm{b})] if $\delta_R$ is linear, then $R$ is unipotent.
\end{itemize}
\end{theorem} 

\begin{proof}
(a) Our claim is clear if $\Z(R)=\{1\}$. So let $1\ne z\in \Z(R)$ and
assume, by contradiction, that $z$ is not unipotent. By Lemma \ref{shape2} up to conjugation $z$
has shape \eqref{shape3}. Lemma \ref{Smith} gives the contradiction that its centralizer  is not transitive on the 
affine vectors.
\smallskip

\noindent (b) Let $\widehat \F$ be the algebraic closure of $\F$. By what observed above, substituting $R$ with
$\widehat R$, if necessary, we may assume $\F$ algebraically closed.
By contradiction, suppose that $z\in R$ is not unipotent.
Up to conjugation we may suppose $z$ as in the statement of Lemma \ref{shape2}.
In the same notation, $\delta_R(w_1,0)=\diag(A_1-I_{m_1},A_2-I_{m_2})$.   
Let $\xi\in \F$ be an eigenvalue of $A_2-I_{m_2}$
and $0\neq v\in \F^{m_2}$ be a corresponding  eigenvector. 
We have  $\xi\ne 0$ and, by  linearity, 
$\delta_R(-\xi^{-1}w_1,0)=\diag\left(-\xi^{-1}\left(A_1-I_{m_1}\right),-\xi^{-1}\left(A_2-I_{m_2}\right)\right)$. 
It follows that the first row of the product $\mu_{R}(0,v)\mu_R(-\xi^{-1}w_1,0)$ is equal to the first row of the second 
factor $\mu_R(-\xi^{-1}w_1,0)$.
From the regularity of $R$ we get $\mu_{R}(0,v)=I_{n+1}$. In particular $v=0$, a 
contradiction. We conclude that  $R$ is unipotent.
\end{proof}

The following examples show how the linearity of $\delta_R$ seems to be necessary to manage a classification 
of the regular subgroups of $\AGL_n(\F)$, even when unipotent.

\begin{ex}\label{degio}
Let $\B$ be a basis of $\R$ over $\Q$. For every  subset $S$ of $\B$ denote by  $f_S:\B\to \R$ the function
such that $f_S(v)=v$ if $v\in S$, $f_S(v)=0$ otherwise. 
Let  $\widehat f_S$ be its extension by linearity to $\R$. In particular $\widehat f_S\in \Hom_\ZZ(\R,\R)$ and
$\Imm(\widehat f_S)=\sp_\R(S)$, the subspace
generated by $S$. 

Consider the regular subgroup $R_S$ of the affine group $\AGL_2(\R)$ defined by:
$$R_S=\left\{\begin{pmatrix}
1&x&y\\
0&e^{\widehat f_S(y)}&0\\
0&0&1
\end{pmatrix}: x,y\in \R\right\}.$$
The set of eigenvalues of the matrices in $R_S$ is  $\left\{e^r : r\in \sp_\R(S) \right\}$.
It follows that $S_1\ne S_2$ gives $R_{S_1}$ not conjugate to $R_{S_2}$ under $\GL_3(\R)$. A fortiori
$R_{S_1}$ and $R_{S_2}$ are not conjugate under $\AGL_2(\R)$.
Thus in $\AGL_2(\R)$ there are as many  non conjugate regular subgroups as possible, namely
$2^{|\R|}$.
\end{ex}

\begin{ex}\label{add}
Let $\F\neq \F_p$ be a field of characteristic $p>0$.
Then the set
$$R=\left \{ \begin{pmatrix}
          1 & x_1 & x_2 \\ 
          0 & 1   & x_1^p\\
          0 &  0 & 1
           \end{pmatrix} : x_1,x_2\in \F
\right\}$$
is a non abelian unipotent regular subgroup of $\AGL_2(\F)$ such that $\delta_R$ is not linear.
\end{ex}

\begin{ex}[Heged\H{u}s, \cite{H2000}] \label{Heg}
Let $n\ge 4$ and $\F=\F_p$ ($p$ odd). Take the  matrices 
$A=\diag(A_3,I_{n-4})$ and $J=\diag(J_3,I_{n-4})$ of $\GL_{n-1}(\F_p)$, where 
$A_3=\left(
\begin{smallmatrix}
1 & 2 & -2 \\ 0 & 1 & -2 \\ 0 & 0 & 1
\end{smallmatrix}\right)$
and $J_3=\left(\begin{smallmatrix}
0 & 0 & 1 \\ 0 & 1 & 0 \\ 1 & 0 & 0
\end{smallmatrix}\right)$.
Then $A$ has order $p$, its minimum polynomial has degree $3$ and $AJA^{\T}=J$.
The subset $R$ of order $p^4$, defined by
\begin{equation*}
R=\left\{\begin{pmatrix}
1& v & h+\frac{vJv^{\T}}{2}\\
0 & A^h & A^hJ v^{\T}\\
0& 0 &1
\end{pmatrix}: v\in {\F_p}^{n-1}, h\in \F_p \right\},
\end{equation*}
is a regular submonoid of $\AGL_n(\F_p)$, hence a regular subgroup.
Moreover $R\cap \Tr=\left\{1\right\}$.
Note that  $\delta_R$ is not linear: suppose the contrary and observe that the elements
$\mu_R(v_n)$ and $\mu_R(2v_n)$ correspond to $h=1$ and $h=2$ respectively ($v_n=(0,\ldots,0,1)$).
Thus, in order that $\delta_R(2v_n)= 2\delta_R(v_n)$ we should have $A^2-I_{n-1}=2A-2I_{n-1}$,
i.e. $(A-I_{n-1})^2=I_{n-1}$, against our choice of $A$.
\end{ex}

\section{Algebras and regular subgroups}\label{sec:alg}

In this section we highlight the connections between regular  subgroups and finite dimensional split local algebras 
over a field $\F$. 
To this purpose we recall that an $\F$-algebra $\A$ with $1$ is called \emph{split local} if 
$\A/ \J(\A)$ is isomorphic to $\F$, where $\J(\A)$ denotes the Jacobson radical of $\A$.
In particular $\A=\F\, 1 + \J(\A)$, where the set $\F\,1=\{\alpha 1_{\A}: \alpha\in \F\}$ is a 
subring of $\A$  contained in its centre $\Z(\A)$. Note that  $\A\setminus \J(\A)$ is the set $\A^\ast$ of the 
invertible elements of  $\A$. 
If $\J(\A)$, viewed as an $\F$-module, has finite dimension, we say that $\A$ is {\em finite dimensional}.

If $\psi$ is an isomorphism  between two local split $\F$-algebras $\A_1, \A_2$, 
then $\psi(\J(\A_1))=\J(\A_2)$ and $\psi(\alpha\,1_{\A_{1}})=\alpha 1_{\A_2}$ for all $\alpha \in \F$.

\begin{theorem} \label{th:Poo}
Let $\A$ be a  finite dimensional split local $\F$-algebra. In the above notation, set $n=\dim_\F(\J(\A))$. 
Then, with respect to the product in $\A$, the subset 
$$R=1 + \J(\A)=\{1+v: v \in \J(\A)\}$$
is a group, isomorphic to a regular subgroup of $\AGL_{n}(\F)$ for which $\delta_R$ is linear.
\end{theorem}

\begin{proof}
Clearly $R$ is closed under multiplication. Moreover any $r\in R$ has an inverse in $\A$.
From $r-1\in \J(\A)$ we get $r^{-1}-1\in\J(\A)$, whence $r^{-1}\in R$.
We conclude that $R$ is a group. 
Consider the  map 
$R\to \AGL_{n}(\F)$ such that, for all $v\in \J(\A)$:
\begin{equation}\label{ISO}
1+v\mapsto \begin{pmatrix}
1& v_{\B}\\
0&I_n+\delta_R(v)
\end{pmatrix},
\end{equation}
where $v_{\B}$ and $\delta_R(v)$ are, respectively, the coordinate vector of $v$ and the matrix 
of the right multiplication by $v$ with respect to a fixed basis $\B$ 
of $\J(\A)$, viewed as $\F$-module. In particular, considering the right multiplication by $w$:
\begin{equation}\label{multiplication}
\left(vw\right)_{\B}=v_\B\delta_R(w),\quad \textrm{ for all } v,w\in \J(\A).
\end{equation}
The map \eqref{ISO} is injective and we claim that it is a group monomorphism.

Set  $\delta=\delta_R$ for simplicity. For all $v,w\in \J(\A)$ we have $\delta (vw)=\delta(v)\delta(w)$
 by the associativity law and $\delta (v+w)=\delta(v)+\delta(w)$ 
by the distributive laws.
Now:
$$(1+v)(1+w)=1+v+w+vw\ \mapsto \ \begin{pmatrix}
1&v_{\B}+w_{\B}+ \left(vw\right)_{\B}\\
0&I_n+\delta(v)+\delta(w)+\delta (vw)
\end{pmatrix}.$$
On the other hand, considering the images of $1+v$ and $1+w$:
$$\begin{pmatrix}
1& v_{\B}\\
0&I_n+\delta(v)
\end{pmatrix}\begin{pmatrix}
1& w_{\B}\\
0&I_n+\delta(w)
\end{pmatrix}=\begin{pmatrix}
1& w_{\B}+v_{\B}+v_{\B}\delta(w)\\
0&I_n+\delta(v)+\delta(w)+\delta(v)\delta(w)
\end{pmatrix}.$$
Thus \eqref{ISO} is a homomorphism if and only if \eqref{multiplication} holds. 
This proves our claim.
\end{proof}

Notice that Theorem \ref{th:Poo} shows how to construct a regular subgroup starting from the presentation of a split 
local algebra.

\begin{ex}\label{Poon}
As in \cite{P}, consider the split local algebra 
$$\A= \frac{\F[t_1,t_2,t_3]}{\langle {t_1}^2+{t_2}^2, {t_1}^2+{t_3}^2, t_1 t_2, t_1t_3, t_2t_3\rangle }.$$
In this case $\J(\A)\cong \F^4$ has a basis given by $\{t_1, {t_1}^2, t_2,t_3\}$ and considering the right 
multiplication by an 
element of this basis we get $\delta(t_1)=E_{1,2}$, $\delta({t_1}^2)=0$,
$\delta(t_2)=-E_{3,2}$, $\delta(t_3)=-E_{4,2}$, 
 and so the associated regular subgroup is
$$R=\left\{
\begin{pmatrix}
 1 & x_1 & x_2 & x_3 & x_4 \\ 
  0 & 1 & x_1 &0& 0 \\ 
 0 & 0 & 1 & 0 & 0 \\ 
 0 & 0& -x_3 &1 & 0  \\ 
0 & 0 &-x_4 & 0  & 1
  \end{pmatrix}: x_1,x_2,x_3,x_4\in \F
\right\}.$$ 
\end{ex}

Conversely we have the following result. 

\begin{theorem}\label{taulin}
Let  $R=\mu_R(\F^n)$ be a regular subgroup of $\AGL_n(\F)$.
Set $$V=R-I_{n+1}= \left\{\begin{pmatrix} 0& v\\ 0 & \delta_R(v)   \end{pmatrix}\mid  v\in 
\F^n\right\}\quad \textrm{ and } \quad\A_R=\F I_{n+1}+ V.$$
\begin{itemize}
\item[({\rm a})] If $\delta_R$ is additive, then $V$ is a subring (without identity) of $\Mat_{n+1}(\F)$;
\item[({\rm b})] the function $\delta_R$ is linear if and only if  $\A_R$ is a split local subalgebra of 
$\Mat_{n+1}(\F)$
with $\J(\A_R)=V$.
\end{itemize}
\end{theorem}

\begin{proof}
 Set $\mu=\mu_R$, $\delta=\delta_R$, $I=I_{n+1}$ and $\A=\A_R$.

\noindent (a) Since $\delta$ is additive, $V$ is an additive subgroup.
From $(\mu(v)-I)(\mu(w)-I)= \left(\mu(v)\mu(w)-I\right)-\left(\mu(v)-I\right)-\left(\mu(w)-I\right)$
it follows that $V$ is a subring.
\smallskip

\noindent (b) Suppose first that $\delta$ is linear. Using (a) we have that $\A$  is an additive subgroup.
By the linearity, $(\F I) V=V(\F I)=V$. It follows that  $\A \A=\A$,
hence $\A$ is a subalgebra of $\Mat_{n+1}(\F)$. Again linearity gives
$\alpha I+ \left(\mu(\alpha v)-I\right)=\alpha \mu(v)$ for all $\alpha\in \F$, $v\in \F^n$.
Thus $\A\setminus V=\F^\ast R$ consists of elements with inverse in $\A$.
We conclude that  $V=\J(\A)$ and $\A$ is a split local $\F$-algebra. 

Vice-versa, let $\A$ be a split local subalgebra with $\J(\A)=V$. In particular
$V$ is an additive subgroup, whence  $\delta(v+w)=\delta(v)+\delta(w)$ for all $v,w\in \F^n$.
Since $V$ is an ideal we get $\delta(\alpha v)=\alpha \delta(v)$ for all $\alpha\in \F$, $v\in \F^n$.
\end{proof}

Our classification of the  regular subgroups of $\AGL_n(\F)$ is based on the following proposition (see \cite[Theorem 
1]{CDS}),
where $\widehat{\GL}_n(\F)$ is defined as in the Introduction.

\begin{proposition}\label{prop:algebre}
Assume that $R_1,R_2$ are regular subgroups of $\AGL_n(\F)$ such that $\delta_{R_1}$ and $\delta_{R_2}$ are
linear maps.
Then the following conditions are equivalent:
\begin{itemize}
\item[(\rm{a})] $R_1$ and $R_2$ are conjugate in $\widehat{\GL}_n(\F)$;
\item[(\rm{b})] $R_1$ and $R_2$ are conjugate in $\AGL_n(\F)$;
\item[(\rm{c})] the algebras $\A_{R_1}$ and $\A_{R_2}$  are isomorphic.
\end{itemize}
\end{proposition}

\begin{proof}
Set $\delta_1=\delta_{R_1}$, $\A_1=\A_{R_1}$, $\delta_2=\delta_{R_2}$, $\A_2=\A_{R_2}$, $\mu_1=\mu_{R_1}$ and 
$I=I_{n+1}$.

\noindent $(a)\Longrightarrow (b) \Longrightarrow (c)$ is  clear.
Let us prove $(c) \Longrightarrow (a)$. By Theorem \ref{taulin}, $\A_1$ 
and $\A_2$ are split local algebras with Jacobson radicals $R_1-I$ and $R_2-I$, respectively.
Suppose that $\psi: \A_1\to \A_2$ is an algebra isomorphism.
In particular $\psi(R_1-I)=R_2-I$ and $\psi$ induces an $\F$-automorphism  of $\F^n$.
Let $P\in \GL_n(\F)$ be the matrix of this automorphism with respect to the canonical basis of $\F^n$.
Then 
$$\psi\begin{pmatrix} 0 & v \\ 0 &\delta_1(v)   \end{pmatrix}=
\begin{pmatrix} 0 & vP \\ 0 & \delta_2(vP)   \end{pmatrix} .$$
For all $v,w\in \F^n$ we have 
$\psi(\mu_1(w) \mu_1(v))=\psi(\mu_1(w))\psi(\mu_1(v))$. This implies
$\delta_1(v)P=P\delta_2(vP)$, whence
$\delta_2(vP)=P^{-1}\delta_1(v)P$  for all $v\in \F^n$.
We conclude: 
$$\psi\begin{pmatrix} 0& v \\ 0 & \delta_1(v)   \end{pmatrix}=
\begin{pmatrix}  0& vP \\ 0 & \delta_2(vP)   \end{pmatrix}= 
\begin{pmatrix}
1&0\\
0&P\end{pmatrix}^{-1}\begin{pmatrix} 0& v \\ 0 & \delta_1(v)   \end{pmatrix}\begin{pmatrix}
1&0\\
0&P\end{pmatrix}.$$
\end{proof}

\section{Centralizers of unipotent elements}\label{Sec:unip}

Our classification of unipotent regular subgroups of $\AGL_n(\F)$ is connected to the
classical theory of canonical forms of matrices. For the reader's convenience we recall the relevant facts.
For all $m\ge 2$, the conjugacy classes of unipotent elements in $\GL_{m}(\F)$ are parametrized by the Jordan canonical 
forms
\begin{equation}\label{J}
J=\blockdiag(J_{m_1}, \dots, J_{m_k}), \quad m=\sum m_i,
\end{equation}
where each $J_{m_i}$ is a Jordan block of size $m_i$, namely a
matrix having all $1$'s on the main diagonal and the diagonal above it,
and $0$'s elsewhere. 
A  Jordan block $J_{m}$ has minimal polynomial $(t-1)^m$ and its 
centralizer is the $m$-dimensional algebra 
\begin{equation}\label{Tmm}
T_{m,m}(\F):=\left\{
\left(
\begin{array}{cccccc}
x_0 & x_1  & \ldots & x_{m-2} &x_{m-1}\\ 
0 & x_0 &  \dots&  x_{m-3} & x_{m-2}\\
 &    & \ddots & & \vdots\\
0 & 0&  \ldots & x_0 & x_1\\
0 & 0&  \ldots & 0 & x_0
\end{array}\right): x_i\in \F\right\}
\end{equation}
generated by  $J_m$.
To study the centralizer of $J$ in \eqref{J} we write  $c\in \Mat_m(\F)$ as:
\begin{equation}\label{C}
c= 
\begin{pmatrix}
C_{1,1}&\dots &C_{1,k}\\
\vdots &\ddots&\vdots\\
C_{k,1}&\dots &C_{k,k}
\end{pmatrix},\ C_{i,j}\in \Mat_{m_i,m_j}(\F).
\end{equation}

Clearly $c$  centralizes $J$ if and only if
\begin{equation}\label{centra}
J_{m_i}C_{i,j}= C_{i,j}J_{m_j},\quad \textrm{for all } i,j.
\end{equation}

\begin{lemma}\label{centralizzatore} 
Take  $J$ as in \eqref{J} and assume further that 
$m_1\ge \dots \ge m_k$. 
Then the element
$c$ above centralizes $J$ if and only if   
$$C_{i,j}\in T_{m_i,m_j}(\F),\quad \textrm{for all } i,j,$$
where each $T_{m_i,m_i}(\F)$ is defined as in \eqref{Tmm} with $m=m_i$ and, for  $m_i> m_j$:
$$
T_{m_i,m_j}(\F):= \begin{pmatrix}
T_{m_j,m_j}(\F)\\
\noalign{\medskip}
 0\\
\vdots\\
 0
\end{pmatrix},\quad T_{m_j,m_i}(\F):= \begin{pmatrix}
0&\dots & 0&T_{m_j,m_j}(\F)
\end{pmatrix}.
$$
\end{lemma}

\begin{lemma}\label{Jacobson} 
Take $J$ as in \eqref{J}.
If $m_1<m_i$ for some $i\ge 2$, then the group
$\C_{\Mat_{m}(\F)}(J)\cap \AGL_{m-1}(\F)$ 
is not transitive 
on affine row vectors $(1,x_1,\dots, x_{m-1})$. 
\end{lemma}

\begin{proof}
Consider $c\in \C_{\Mat_{m}(\F)}(J)$ and decompose it as in \eqref{C}. Application of 
\eqref{centra} and elementary matrix calculation give that, whenever  $m_1<m_i$,  the first
row of the matrix $C_{1,i}$ must be the zero vector. Our claim follows immediately.
\end{proof}

\begin{lemma}\label{transitivity}
Let $J$ be as in \eqref{J}. If $k>1$, assume further that $m_1\ge m_i$
for all $i\ge 2$.
Then, for every $v\in \F^{m}$, there exists $c\in \C_{\Mat_{m}(\F)}(J)$ 
having $v$ as first row. If the first coordinate of $v$ is non zero,
such $c$ can be chosen nonsingular. 
\end{lemma}

\begin{proof}
Both claims are direct consequences of the description of $T_{m,m}(\F)$ preceding \eqref{Tmm}
and Lemma \ref{centralizzatore}. 
\end{proof}

\begin{theorem}\label{Jordan form}
Let $R$ be a  regular subgroup of $\AGL_n(\F)$
and $1\ne z$ be an element of the center $\Z(R)$ of $R$. Then, up to conjugation of $R$ under $\AGL_n(\F)$,
we may suppose that $z=J_z$, where $J_z= \blockdiag\left(J_{n_1}, \dots, J_{n_k}\right)$ 
is the Jordan form of $z$ having Jordan blocks of respective sizes  $n_i\ge n_{i+1}$ for all $i\ge 1$.
\end{theorem}

\begin{proof} 
Let $t$ be the number of non-trivial invariant factors of $z$.
\smallskip

\noindent\textbf{Case $t=1$}, i.e. $J_z=J_{n+1}$.
Let $g\in \GL_{n+1}(\F)$ be  such that $g^{-1}zg=J_z$. 
Since $\left\langle e_0^{\T}\right\rangle$ is the eigenspace of $J_z$ (acting on the left) 
we have that $\left\langle ge_0^{\T}\right\rangle$ must be  the eigenspace
of $z$  (acting on the left). From $z\in \AGL_n(\F)$ it follows $ze_0^{\T}=e_0^{\T}$, hence $ge_0^{\T}=\lambda 
e_0^{\T}$. We conclude that 
$\lambda^{-1}g\in \AGL_n(\F)$ conjugates $z$ to $J_z$. 
\smallskip

\noindent\textbf{Case $t>1$}.
By the unipotency of $z$, there exists $g\in \widehat{\GL}_n(\F)$ that
conjugates $z$ to 
$$z'= \begin{pmatrix}
1&w_1&\dots &w_h\\
0&J_{m_1}&&\\
\vdots &&\ddots\\
0&0&\dots&J_{m_h}
\end{pmatrix}.$$
We claim that the first coordinate of  $w_i\in \F^{n_i}$ cannot be $0$ for all $i\ge 1$. 
Indeed, in this case, there exists $u_i\in \F^{n_i}$ such that
$w_i=u_i\left(I_{n_i}-J_{n_i}\right)$ for all $i\ge 1$. 
Setting $u=\left(u_1,\dots, u_h\right)$
we have  
$$z'':= \begin{pmatrix}1&u\\
0&I
\end{pmatrix} z'\begin{pmatrix}1&u\\
0&I
\end{pmatrix}^{-1}= \begin{pmatrix}
1&0 \\
0&J
\end{pmatrix},\quad J=\blockdiag\left(J_{m_1}, \dots,J_{m_h}\right).$$
It follows that $\C_{\AGL_n(\F)}(z'')$ is not transitive on the affine vectors 
by Lemma \ref{Jacobson}, noting that  $J\ne I_{n}$: a contradiction.
So there exists some $t$, $1\le t\le h$,  such that the first coordinate of $w_t$ is non zero.
Up to conjugation  by an obvious permutation matrix in $\widehat{\GL}_n(\F)$ we may assume $t=1$, i.e., $w_t=w_1$.

By Lemma \ref{transitivity} there exists
$p\in \GL_n(\F)$ which centralizes $\blockdiag\left(J_{m_1}, \dots, J_{m_h}\right)$ and   
has $v= (w_1,\dots, w_h)$ as first row. It follows that $vp^{-1}=\left(1,0,\dots, 0\right)$. 
Thus
$$z'''= \begin{pmatrix}
1&0\\
0&p
\end{pmatrix} z'\begin{pmatrix}1&0\\
0&p^{-1}
\end{pmatrix}= \blockdiag(J_{1+m_1},J_{m_2}, \dots, J_{m_h}).$$
Again by Lemma \ref{Jacobson} we must have
$m_1+1\ge m_i$ for all $i\ge 2$. A final conjugation, if necessary, by a permutation matrix 
in $\widehat{\GL}_n(\F)$ allows to arrange the blocks of $z'''$ in non-increasing sizes,
i.e. allows to conjugate $z'''$ to $J_z$ as in the statement.
\end{proof}

\section{Some useful parameters}\label{param}

We introduce some parameters that will be used mainly to exclude conjugacy among regular subgroups.
Let $H$ be any unipotent subgroup of $G_n:=\AGL_n(\F)$.
For each $h-I_{n+1}\in H-I_{n+1}$ we may consider the degree of its  minimal polynomial over $\F$,
denoted by $\deg \min_{\F}(h-I_{n+1})$, and its rank, denoted by $\rk(h-I_{n+1})$.
E.g., $\deg \min_{\F}(h-I_{n+1})=n+1$ if and only if $h$ is conjugate to a unipotent Jordan block  $J_{n+1}$ of size $n+1$. Note that $\rk(J_{n+1}-I_{n+1})=n$.

Hence,  we may set:
\begin{equation*}%\label{d(U)}
\begin{array}{rcl}
d(H) & = &\max\{\deg \min_{\F}(h-I_{n+1})\mid  h \in H\}; \\
r(H) & = &\max\{\rk(h-I_{n+1})\mid h \in H \}; \\
k(H) & = &\dim_\F \{w\in \F^n: w\pi(h)=w \}.
\end{array}
\end{equation*}
If $H$ is a subgroup of a regular subgroup $U$ such that $\delta_U$ is linear, then
$k(H)=\dim_\F \Ker(\delta_U |_H)$.

Clearly if two unipotent subgroups $H_1$, $H_2$ are conjugate, then: 
$$d(H_1)=d(H_2),\quad r(H_1)=r(H_2),\quad k(H_1)=k(H_2),$$
$$d(\Z(H_1))=d(\Z(H_2)),\quad r(\Z(H_1))=r(\Z(H_2)),
\quad k(\Z(H_1))=k(\Z(H_2)).$$

\begin{lemma}\label{indecomposable}
Let $H$ be a unipotent subgroup of $\AGL_n(\F)$ and assume that $\F^n$ is the direct sum of non-trivial
 $\pi(H)$-invariant subspaces $V_1, \dots, V_s$.
Then  $k(H)\geq s$.
In particular, if $k(H)=1$, then $H$ is indecomposable.
\end{lemma}

\begin{proof}
$H$ induces on each $V_i$ a unipotent group $H_i$. So, in each $V_i$, 
there exists a non-zero vector $w_i$ fixed by all elements of $H_i$ 
(see \cite[Theorem 17.5 page 112]{H}). 
It follows that $ w_1, \dots, w_s$ are $s$ linearly independent vectors of $\F^n$
fixed by $\pi(H)$. 
\end{proof}

\begin{notaz}
For sake of brevity, we write
\begin{equation}\label{Ubreve}
U=\begin{pmatrix}
1&\begin{array}{cccc}
x_1 & x_2 & \ldots & x_n   
  \end{array}
\\
0&\tau_U(x_1 , x_2,  \ldots,  x_n )
\end{pmatrix}
\end{equation}
to indicate the regular subgroup
$$U=\left\{\begin{pmatrix}
1&\begin{array}{cccc}
x_1 & x_2 & \ldots & x_n   
  \end{array}
\\
0&\tau_U(x_1 , x_2,  \ldots,  x_n )
\end{pmatrix}:  x_1,\ldots,x_n\in \F\right\}.$$
For all $i\le n$, we denote by $X_i$ the matrix of $U-I_{n+1}$ obtained taking $x_i=1$ and $x_j=0$ for all $j\neq i$.\\
The set $\left\{v_1,\dots , v_n\right\}$ is the canonical basis  of $\F^n$.
\end{notaz}

We recall that the center of a unipotent group is non-trivial.

\begin{lemma}\label{Cn} 
Let $U$ be a unipotent regular subgroup of $G_n$. 
If $d(\Z(U))=n+1$ then, up to conjugation,  $U=\C_{G_n}(J_{n+1})$. Moreover $U$ is abelian.
\end{lemma}

\begin{proof} Up to conjugation $J_{n+1}\in \Z(U)$, whence $U\le \C_{G_n}(J_{n+1})$. Since this group is regular, we 
have 
$U=\C_{G_n}(J_{n+1})$.
It follows that $U$ is abelian.
\end{proof}

\begin{lemma}\label{Tn} 
Let $U$ be a regular subgroup of $G_n$ such that $\delta_U$ is linear. 
If $r(U)=1$, then $U$ is  the translation subgroup $\Tr$, which is an abelian normal subgroup of $G_n$. 
Furthermore, $d(\Tr)=2$.  
\end{lemma}

\begin{proof}
By the unipotency, we may always assume that $U$ is upper unitriangular. 
Then, the rank condition gives $\delta(v_i)=0$, for $1\le i\le n$,
whence our claim.
\end{proof}

\begin{lemma}\label{Cn-1} 
Let $U$ be a regular subgroup of $G_n$, $n\ge 2$,  such that $\delta_U$ is linear. 
If $d(\Z(U))=n$ then,  up to conjugation, for some fixed $\alpha\in \F$:
\begin{equation}\label{Ralpha}
U=R_\alpha=
\left(\begin{smallmatrix}
1 & x_1 & x_2 & \dots & x_{n-2} & x_{n-1} & x_n\\ 
0 & 1 & x_1 &  \dots & x_{n-3} & 0 &  x_{n-2}\\
0 & 0 &  1 & \dots & x_{n-4} & 0 & x_{n-3}\\[-5pt]
\vdots&&&\ddots&\vdots & \vdots &\vdots\\[1pt]
0 & 0 & 0  & \dots &   1  & 0 & x_1\\
0 & 0 & 0  & \dots &   0  & 1 & \alpha x_{n-1}\\
0 & 0 & 0  &\dots  &  0 &   0 & 1
      \end{smallmatrix}\right).
\end{equation}
      In particular $U$ is abelian and $r(U)=n-1$.  Furthermore,
$R_0$ and $R_\alpha$ are not conjugate for any  $\alpha\ne 0$,  

If $n\geq 3$, an epimorphism $\Psi: \F[t_1,t_2]\to \F I_{n+1} + R_0$ is obtained setting 
\begin{equation}\label{eq:U0}
\Psi(t_1)  = X_1,\qquad 
\Psi(t_2)  = X_{n-1}.
\end{equation}
In this case we have $\Ker(\Psi)=\langle {t_1}^n, {t_1}^2, t_1t_2 \rangle$.

If $n\geq 4$ is even, then
$R_\alpha$ is conjugate to $R_1$ for any $\alpha\neq 0$ and an epimorphism $\Psi: \F[t_1,t_2]\to \F I_{n+1} + R_\alpha$
is obtained setting 
\begin{equation}\label{eq:Ualpha}
\Psi(t_1)  = \alpha X_1,\qquad \Psi(t_2)=\alpha^{\frac{n-2}{2}} X_{n-1}. 
\end{equation}
In this case $\Ker(\Psi)=\langle {t_1}^{n-1}-{t_2}^{2}, t_1t_2 \rangle$.

If $n$ is odd, write $\alpha\neq 0$ as $\alpha=\lambda \varepsilon^2$, $\lambda,\varepsilon \in \F^\ast$. Then
$R_\alpha$ is conjugate to $R_\lambda$ and an epimorphism $\Psi: \F[t_1,t_2]\to \F I_{n+1} + R_\alpha$
is obtained setting 
\begin{equation}\label{eq:Ualpha odd}
\Psi(t_1)  = \alpha X_1,\qquad \Psi(t_2)=\lambda^{\frac{n-3}{2}}\varepsilon^{n-2} X_{n-1}. 
\end{equation}
In this case $\Ker(\Psi)=\langle {t_1}^{n-1}-\lambda {t_2}^{2}, t_1t_2 \rangle$.
In particular $R_\alpha$ and $R_\beta$ ($\alpha,\beta\in \F^\ast$) are conjugate if and only if $\beta/\alpha$ is a square in $\F^\ast$.
\end{lemma}

\begin{proof}
Up to conjugation we may suppose that $z=\blockdiag(J_{n},J_1)\in \Z(U)$.
The subalgebra generated by $z$ coincides with the set $\left\{\begin{pmatrix} X & 0 \\ 0 & 1 
 \end{pmatrix}: X\in \C_{G_{n-1}}(J_n)\right\}$. This information gives the values of $\delta(v_i)$
for $1\leq i\leq n-1$. From Lemma \ref{centralizzatore}, we get
$\delta(v_{n})=\alpha E_{n,n-1}$. Applying \eqref{vivj}, 
it follows that $U$ is abelian.
Conjugating by the permutation matrix associated to the transposition $(n,n+1)$ we obtain that $U$ is conjugate to $R_\alpha$.

The subgroups $R_0$ and $R_\alpha$ are not conjugate when $\alpha\neq 0$,  
since $k(R_0)=2$ and $k(R_\alpha)=1$. The presentations of the corresponding algebras can be verified by matrix 
calculation.

Assume now that $n$ is odd. If $\alpha=\beta\varepsilon^2\neq 0$, then the subgroups $R_\alpha$ and $R_\beta$ are conjugate in virtue of \eqref{eq:Ualpha odd} and Proposition \ref{prop:algebre}.
Conversely, suppose that $\alpha,\beta\in \F^\ast$ and 
that $Q^{-1} R_\alpha Q=R_\beta$ for some $Q\in\widehat{\GL}_{n}(\F)$ (see Proposition \ref{prop:algebre}).
According to \eqref{Ralpha}, write $$R_\alpha=\begin{pmatrix} 1 & X  & x_n \\ 
0 & I_{n-1}+D_X & AX^{\T}\\
0 & 0 & 1
\end{pmatrix},
\quad
R_\beta=\begin{pmatrix} 1 & Y & y_n \\ 
0 & I_{n-1} +D_Y  & BY^{\T}\\
0  & 0 & 1\end{pmatrix},
$$
where $X=(x_1,\ldots,x_{n-1})$, $Y=(y_1,\ldots,y_{n-1})$, $A=\diag(\antidiag(1,\ldots,1),\alpha)$
and $B=\diag(\antidiag(1,\ldots,1),\beta)$. Hence $\det(A)=\zeta\alpha$ and $\det(B)=\zeta\beta$,
where $\zeta=(-1)^{\frac{n+1}{2}}$.
We may assume  that $v_n Q=\lambda v_n$ since 
$\left\langle v_n\right\rangle$ is the subspace fixed pointwise by  both subgroups.
Thus, for some $N\in \F^{n-1}$ and $\lambda \in \F^\ast$, $Q=Q_1Q_2$, where
$$Q_1=\begin{pmatrix} 
1 & 0 & 0 \\ 0 & I_{n-1} & \lambda^{-1} N \\
 0 & 0& 1
\end{pmatrix},
  \quad 
  Q_2=\begin{pmatrix} 
1 & 0 & 0 \\ 0 & M & 0 \\
 0 & 0& \lambda
\end{pmatrix},
  \quad \det(M)\neq 0,\;
  \lambda \ne 0.$$
Now, $Q_1^{-1}R_\alpha Q_1=\widetilde{R}_\alpha$, where
$$\widetilde{R}_\alpha=
\begin{pmatrix} 
1 & X & \tilde x_n \\ 0 & I_{n-1}+D_X & \lambda^{-1}D_XN+A X^{\T} \\
 0 & 0& 1
\end{pmatrix}=\begin{pmatrix} 
1 & X & \tilde x_n \\ 0 & I_{n-1}+D_X & \widetilde{A} X^{\T} \\
 0 & 0& 1
\end{pmatrix},
$$
with  $\det(\widetilde{A})=\zeta \alpha$.
From $\widetilde{R}_\alpha Q_2-Q_2R_{\beta}$ we get, in particular,
$Y=XM$ and $MBY^{\T}=\lambda \widetilde{A} X^{\T}$, whence
$MBM^{\T} X^{\T}=\lambda \widetilde{A} X^{\T}$, for all $X\in \F^{n-1}$.
It follows that $MBM^{\T}=\lambda \widetilde{A}$ and taking the determinant of both sides, we obtain
$\zeta \beta (\det(M))^2=\zeta\lambda^{n-1} \alpha$.
We conclude that $\beta/\alpha$ is a square. 
\end{proof}

\section{Standard  regular subgroups}\label{Sec:standard}

The aim of this section is to define, for every partition $\lambda$ of $n+1$ different from $(1^{n+1})$, 
one or two abelian regular subgroups $S_\lambda, S_\lambda^{\s}$ of $\AGL_n(\F)$ so that different partitions define 
non-conjugate subgroups.

To this purpose we start by identifying the direct product $\AGL_{m_1}(\F)\times\AGL_{m_2}(\F)$ 
with the stabilizer of $\F^{m_1}$ and $\F^{m_2}$ in $\AGL_{m_1+m_2}(\F)$,
namely with the subgroup:
\begin{equation*}%\label{emb}
\left\{\begin{pmatrix}
1&v&w\\
0&A&0\\
0&0&B
\end{pmatrix}: v\in \F^{m_1},\ w\in \F^{m_2},\ A\in \GL_{m_1}(\F),\ B\in \GL_{m_2}(\F)\right\}.\end{equation*}
Clearly, in this identification, if  $U_i$ are respective regular subgroups of $\AGL_{m_i}(\F)$ for $i=1,2$  
then $U_1\times U_2$ is a regular subgroup of $\AGL_{m_1+m_2}(\F)$.\\

Here, and in the rest of the paper, we denote by $S_{(1+n)}$ the centralizer in $\AGL_n(\F)$ 
of  a unipotent Jordan block of size $n+1$, namely:
\begin{equation*} %\label{U1+n}
S_{(1+n)}= \C_{\AGL_{n}(\F)}(J_{1+n}).
\end{equation*}
Moreover we write  $\tau_{1+n}$ for $\tau_{S_{(1+n)}}$ and $\delta_{1+n}$ for $\delta_{S_{(1+n)}}$.
By Lemma \ref{Cn}, a regular subgroup of $\AGL_n(F)$ is conjugate to $S_{(1+n)}$ if and only if $d(\Z(S_{(1+n)}))=n+1$. 
Observe that $S_{(1+n)}$ is  abelian and  
indecomposable by Lemma \ref{indecomposable}.
More generally, for  a  partition $\lambda$  of $n+1$ such that:
\begin{equation}\label{partition}
\lambda=(1+n_1, n_2, \dots , n_s),\quad s\ge 1,\quad  n_i\ge n_{1+i}\geq 1,\quad 1\le i\le s-1,
\end{equation}
we define an abelian regular subgroup $S_\lambda$ of $\AGL_n(\F)$. Namely we set:
\begin{equation*}%\label{partition1}
S_\lambda= S_{(1+n_1, \ldots, n_s)}= \prod_{j=1}^s S_{(1+n_j)}.
\end{equation*}
Notice that $S_\lambda$ is indecomposable only for $s=1$. In particular,  if $\lambda=(2,1^{n-1})$, then 
$S_\lambda=\Tr$.

\begin{lemma} 
Given a partition  $(1+n_1,n_2)$ of $n+1$, with $1+n_1\ge n_2>1$, set: 
$$S_{(1+n_1,n_2)}^{\s} = \left\{\left(\begin{array}{ccc}
1&u&v\\
0&\tau_{1+n_1}(u)&w\otimes D u^{\T}\\
0&0&\tau_{n_2}(v)\\
\end{array}\right): u\in \F^{n_1}, v\in \F^{n_2}\right\},$$
where $w=(0,\ldots,0,1)\in \F^{n_2}$, $D =\antidiag(1,\dots, 1)\in \GL_{n_1}(\F)$. 
Then $S_{(1+n_1,n_2)}^{\s}$ is an indecomposable regular subgroup of $\AGL_{n}(\F)$ with 
$d(S_{(1+n_1,n_2)}^{\s})=n_1+2$.
\end{lemma}

\begin{proof}
Routine calculation with matrices shows that $S_{(1+n_1,n_2)}^{\s}$ is closed under 
multiplication, hence a subgroup. Moreover it is indecomposable by Lemma \ref{indecomposable}. 
Again by matrix calculation one can see that $d(S_{(1+n_1,n_2)}^{\s})=n_1+2$. 
To check that $S_{(1+n_1,n_2)}^{\s}$ is a subgroup it is useful to note that: 
\begin{itemize}
 \item[(a)] all  components of $u(w\otimes Du^{\T})$, except possibly the last one,  are zero; 
\item[(b)] $\delta_{n_2}(v)$ 
does not depend on the last component of $v$; 
\item[(c)] $(w\otimes Du^{\T})\delta_{n_2}(v)=0$; 
\item[(d)] $w\otimes D(u_1\delta_{1+n_1}(u_2))^{\T}=\delta_{1+n_1}(u_1)(w\otimes Du_2^{\T})$ for all
$u_1,u_2\in \F^{n_1}$.
\end{itemize}
\end{proof}

Next, consider a partition $\mu$ of $n+1$ such that:
\begin{equation}\label{partitions}
\mu=(1+n_1, n_2, \dots , n_s),\ s\ge 2,\  1+n_1\geq n_2>1,\ n_i\ge n_{1+i}\geq 1,\ 2\le i\le s-1.
\end{equation}
We define the abelian regular subgroup $S_\mu^\s$ in the following way:
\begin{equation*}%\label{partition4}
S_\mu^\s = S_{(1+n_1,n_2,n_3 \ldots, n_s)}^{\s}=S_{(1+n_1,n_2)}^{\s}\times  \prod_{j=3}^s S_{(1+n_j)}.
\end{equation*}

The regular subgroups $S_\lambda,S_\mu^{\s}$ associated to partitions as above  will be 
called standard regular subgroups. As already mentioned  $S_\lambda$ and $S_\mu^{\s}$ are always abelian.

\begin{rem}\label{dim}
Let $\lambda =(1+n_1, \dots , n_s)$ be a partition as in \eqref{partition}.
Then $\F^n$ is a direct sum of $s$  indecomposable modules  of respective dimensions $n_i$ for which 
$d=n_i+1$.
Furthermore,
\begin{equation}\label{presUlam}
\A_\lambda=\F I_{n+1}+S_\lambda\cong \frac{\F[t_1,t_2,\ldots,t_s]}{\langle t_i^{n_i+1}: 1\leq i\leq 
s;\; t_i t_j: 1\leq i<j\leq s \rangle}.
\end{equation}
An epimorphism $\Psi: \F[t_1,t_2,\ldots,t_s]\to \A_\lambda$ is obtained by setting 
\begin{equation}\label{Psilam}
\Psi(t_i)=X_i.
\end{equation}

Let now $\mu=(1+n_1, \dots , n_s)$ be a partition as in \eqref{partitions}.
Then $\F^n$ is a direct sum of $s-2$ indecomposable  modules of respective dimension $n_i$, $i\ge 3$,
for which $d=n_i+1$, and a single indecomposable module of dimension 
$n_1+n_2$ for which $d=n_1+2$.
Furthermore,
\begin{equation}\label{presUmu}
\A_\mu^\s=\F I_{n+1}+ S_\mu^\s\cong \frac{\F[t_1,t_2,\ldots,t_s]}{\langle 
t_1^{n_1+1}-t_2^{n_2};\; t_i^{n_i+1}: 3\leq i\leq s; \;t_i t_j: 1\leq i<j\leq s \rangle}.
\end{equation}
An epimorphism $\Psi: \F[t_1,t_2,\ldots,t_s]\to \A_{\mu}^\s$ is obtained by setting 
\begin{equation}\label{Psimu}
\Psi(t_i)=X_i.
\end{equation}
\end{rem}

\begin{theorem} 
Let  $\lambda_1 =(1+n_1, \dots , n_s),\lambda_2=(1+m_1, \dots , m_t)$ be two partitions of $n+1$ as in 
\eqref{partition} and let $\mu_1 =(1+a_1, \dots ,a_h),\mu_2 =(1+b_2, \dots,b_k)$ be two partitions of $n+1$ as in 
\eqref{partitions}.
Then 
\begin{itemize}
\item[(\rm{a})] $S_{\lambda_1}$ is not conjugate to $S_{\mu_1}^\s$;
\item[(\rm{b})] $S_{\lambda_1}$ is conjugate in $\AGL_{n+1}(\F)$ to $S_{\lambda_2}$ if and only if 
$\lambda_1=\lambda_2$;
\item[(\rm{c})] $S_{\mu_1}^\s$ is conjugate in $\AGL_{n+1}(\F)$ to $S_{\mu_2}^\s$ if and only if $\mu_1=\mu_2$.
\end{itemize}
\end{theorem}

\begin{proof}
(a) There exists an isomorphism $\varphi: \A_{\lambda_1} \to \A_{\mu_1}^\s$.
Setting $vu:=v\varphi(u)$ for all $u\in \A_{\lambda_1}$, $v\in \J(\A_{\mu_1}^\s)$ we may consider $\J(\A_{\mu_1}^\s)$ 
as an $\A_{\lambda_1}$-module.
Clearly $\J(\A_{\lambda_1}) \cong \J(\A_{\mu_1}^\s)\cong \F^n$ as $\A_{\lambda_1}$-modules. By construction, $\F^n$ is 
a 
direct
sum of indecomposable $\A_{\lambda_1}$-modules. 
By Remark \ref{dim}, for $\J(\A_{\lambda_1})$ each direct summand has dimension $n_i$ and $d=n_i+1$ and 
for $\J(\A_{\mu_1}^\s)$ there is a direct summand of dimension $a_1+a_2$ and $d=a_1+2< a_1+a_2 +1$. 
Hence, by Krull-Schmidt Theorem (e.g., see \cite[page 115]{J}), $S_{\lambda_1}$ is not conjugate to $S_{\mu_1}^\s$.

(b) Arguing as before, we may consider  $\J(\A_{\lambda_2})$ 
as an $\A_{\lambda_1}$-module.
In this case, by construction, $\F^n$ is a 
direct
sum of indecomposable $\A_{\lambda_1}$-modules of dimension $n_i$ such that $d=n_i+1$ (see Remark \ref{dim}(1)).
By Krull-Schmidt Theorem they are conjugate if and only if $s=t$ and $n_i=m_i$.

(c) Again, we may consider  $\J(\A_{\mu_2}^\s)$ 
as an $\A_{\mu_1}^\s$-module.
In this case, by construction, $\F^n$ is a 
direct
sum of indecomposable $\A_{\mu_1}^\s$-modules, one of dimension $a_1+a_2$ such that $d=a_1+2<a_1+a_2+1$, the others of 
dimension $n_i$ such that $d=a_i+1$ (see Remark \ref{dim}(2)).
By Krull-Schmidt Theorem  they are conjugate if and only if $h=k$ and $a_i=b_i$.
\end{proof}

\begin{rem}\label{rem01}
Note that for $n\geq 3$, the regular subgroups $R_0$ and $R_1$ of Lemma \ref{Cn-1} coincide, respectively, with $S_{(n,1)}$ and  
$S_{(n-1,2)}^\s$.
\end{rem}

\section{Regular subgroups with linear $\delta$}\label{low}

By \cite[Lemma 5.1]{T}, the only regular subgroup of $\AGL_1(\F)$ is the translation subgroup 
$\Tr=S_{(2)}=\begin{pmatrix}
x_0 & x_1\\
0 & x_0
\end{pmatrix}$.  
However, already for $n=2$, a
description becomes much  more complicated.   
As seen in Example \ref{degio} there are $2^{|\R|}$ conjugacy classes of regular subgroups of $\AGL_2(\R)$
with trivial center. Also, restricting to the unipotent case, one may only say that every unipotent 
regular subgroup is conjugate to:
$$\begin{pmatrix} 1 & x_1 & x_2 \\ 0 & 1 & \sigma(x_1) \\ 0 & 0 & 1   \end{pmatrix}$$
where  $\sigma\in \Hom_{\mathbb{Z}}(\F,\F)$ (\cite[Lemma 5.1]{T}). See also Examples \ref{add} and \ref{Heg}.

Thus, from now on, we restrict our attention to regular subgroups $U$ of 
$\AGL_n(\F)$ such that the map $\delta_U$ defined in \eqref{tauv} is linear.
$U$ is unipotent by Theorem \ref{lin->unip} and so there exists $1\ne z\in \Z(U)$.
By Theorem \ref{Jordan form}, up to conjugation, we may assume that $z$ is a Jordan form.
For the notation we refer to Section \ref{param}. Our classification is obtained working  on the parameters $d$, $r$ 
and $k$, considered in this order. 

In the tables of the next subsections the indecomposability of the regular subgroups follows from Lemma \ref{indecomposable}, since $k(U)=1$, except for the subgroup $U_1^4$ of Table \ref{Tab6}, for which we refer to Lemma \ref{J3J2}. Also, we describe the kernel of the epimorphism 
$$\Psi: \F[t_1,\ldots,t_s]\rightarrow \F I_{n+1}+U.$$

\subsection{Case $n=2$}

If $d(\Z(U))=3$, then $U$ is abelian and  conjugate to $S_{(3)}$ by Lemma \ref{Cn}.
If $d(\Z(U))=2$, then   $U$ is abelian and $r(U)=1$ by  Lemma \ref{Cn-1}.
Thus $U$ is conjugate  $S_{(2,1)}$ by Lemma \ref{Tn}.
Table \ref{Tab4} summarizes these results.

\begin{table}[!th]
\begin{tabular}{ccccc}
  $U$ & $\F I_{n+1} + U$ & $\Psi$ & $\Ker(\Psi)$ \\\hline
\noalign{\smallskip}
 $S_{(3)}$ & $\begin{pmatrix}
x_0 &x_1&x_2\\
0 &x_0&x_1\\
0 & 0&x_0
\end{pmatrix}$ & \eqref{Psilam} &  $\langle {t_1}^3\rangle$ & indec.\\
\noalign{\smallskip}
 $S_{(2,1)}$ & $\begin{pmatrix}
x_0&x_1&x_2\\
0&x_0&0\\
0&0&x_0
\end{pmatrix}$ & \eqref{Psilam} & $\langle t_1, t_2\rangle^2$ \\
\noalign{\smallskip}
\end{tabular}
\caption{Representatives for the conjugacy classes of regular subgroups $U$ of $\AGL_2(\F)$ with linear 
$\delta$, for any field $\F$.}\label{Tab4}
\end{table}

\subsection{Case $n=3$}
To obtain the full classification, we need some preliminary results.

\begin{lemma}\label{J2J2}
If $d(\Z(U))=r(\Z(U))=2$, then $U$ is abelian, conjugate to 
$$U_1^3=\begin{pmatrix}
1&x_1&x_2&x_3\\
0&1&0&x_2\\
0&0&1&x_1\\
0&0&0&1
\end{pmatrix}$$
 and $\char \F=2$. An epimorphism $\Psi: \F[t_1,t_2]\to \F I_{4} + U$ is obtained by setting 
\begin{equation}\label{eq:22}
\Psi(t_1)=X_1,\quad  \Psi(t_2)=X_2.
\end{equation}
In this case, we have $\Ker(\Psi)=\langle {t_1}^2, {t_2}^2\rangle$.
\end{lemma}

\begin{proof}
We may assume $z=\blockdiag(J_2 ,J_2)\in \Z(U)$, whence $\delta(v_1)=E_{2,3}$.
From Lemma \ref{centralizzatore} and the unipotency of $U$ we obtain
$\delta(v_2)=E_{1,3}+\alpha E_{2,1}+\beta E_{2,3}$. It follows $v_1\delta(v_2)=v_3$.
Now, we apply \eqref{prodotto} to $v_1,v_2$, which gives $\delta(v_3)=\delta(v_1)\delta(v_2)=0$. 
Direct calculation shows that $U$ is abelian. Hence
$d(\Z(U))=d(U)=2$. In particular, $(\mu(v_2)-I_4)^2=0$ gives $\alpha=\beta=0$. Finally
$(\mu(v_1+v_2)-I_4)^2=0$ gives  $\char \F=2$. 
\end{proof}

It is convenient to denote by $V(\alpha_2,\alpha_3,\beta_2,\beta_3)$ the regular subgroup defined by:
\begin{equation}\label{U4}
V(\alpha_2,\alpha_3,\beta_2,\beta_3)=\begin{pmatrix}
1&x_1&x_2&x_3\\
0&1&0&0\\
0&\alpha_2 x_2+\alpha_3 x_3&1& 0\\
0&\beta_2 x_2+\beta_3 x_3&0&1
\end{pmatrix}.
\end{equation}
Notice that  $V(\alpha_2,\alpha_3,\beta_2,\beta_3)$ is abelian if and only if  $\alpha_3= \beta_2$.

We need the cosquare $A^{-\T}A$ of a nonsingular matrix $A$. If $A,B$ are congruent, 
i.e., $B=PAP^{\T}$ for a nonsingular $P$,  their cosquares are conjugate (e.g., see \cite{HS}).

\begin{lemma}\label{cosquare}
Let $\beta, \gamma \in \F$ with $\beta\gamma\ne 0$. The subgroups $V_\beta= V(1,1,0,\beta)$ and $V_\gamma 
=V(1,1,0,\gamma)$ are conjugate in $\AGL_3(\F)$ if and only if $\beta=\gamma$.
\end{lemma}

\begin{proof}
Suppose that  $Q^{-1}V_\beta Q=V_\gamma$ for some $Q\in \widehat{\GL}_3(\F)$ (see Proposition \ref{prop:algebre}). 
We may assume that $v_1Q=\lambda v_1$ since 
$\left\langle v_1\right\rangle$ is the subspace of $\F^3$ fixed pointwise by  both subgroups.
Thus $Q$ has shape:
$$Q=\begin{pmatrix}
1&0&0&0\\
0&\lambda&0&0\\
0&q_1 &p_{1,1}&p_{1,2}\\
0&q_2 &p_{2,1}&p_{2,2}
\end{pmatrix}, \quad \lambda \ne 0, \quad P=\begin{pmatrix}
p_{1,1}&p_{1,2}\\
p_{2,1}&p_{2,2}
\end{pmatrix} \textrm{ nonsingular}.$$
The matrix $K=I_4-\frac{q_1}{\lambda}E_{3,2}-\frac{q_2}{\lambda}E_{4,2}$
normalizes $V_\beta$. Hence, substituting $Q$ with $KQ$, we may suppose 
$q_1=q_2=0$.
Setting $B=\begin{pmatrix}
1&1\\
0&\beta
\end{pmatrix}$ and $C=\begin{pmatrix}
1&1\\
0&\gamma
\end{pmatrix}$ we have:
$$V_\beta=\left(\begin{array}{ccc}
1 & x_1 & X \\
0 & 1 & 0 \\ 
0& B X^{\T} &I_2
\end{array}\right)\ X=(x_2,x_3), \quad V_\gamma=
\left(\begin{array}{ccc}
1 & y_1 & Y \\
0 & 1 & 0 \\ 
0& C Y^{\T} &I_2
\end{array}\right)\ Y=(y_2,y_3).$$
From $V_\beta Q =Q V_\gamma$ we get $Y=XP$ and $\lambda BX^{\T} =PCY^{\T}$, whence
$$\lambda BX^{\T}= PCP^{\T}X^{\T}, \quad \textrm{for all } X\in\F^2.$$
It follows that $\lambda B=PCP^{\T}$, i.e., the matrices $\lambda B$ and $C$ are
congruent. So their cosquares must be conjugate. But the 
characteristic polynomials of the cosquares are respectively $t^2+(\beta^{-1}-2)t+1$ and $t^2+(\gamma ^{-1}-2)t+1$, 
whence $\beta=\gamma$.
\end{proof}

\begin{lemma}\label{J2:3}
Suppose $d(\Z(U))=2$ and $r(\Z(U))=1$.
If $U$ is abelian, then $U=S_{(2,1^2)}$. Otherwise,  $U$ is conjugate to exactly one of the following subgroups:
\begin{itemize}
\item[(\rm{a})]
 $N_1=\begin{pmatrix}
1&x_1&x_2&x_3\\
0&1&0&0\\
0&0&1& x_1\\
0&0&0&1
\end{pmatrix}$,  if $k(U)=2$;
 \item[(\rm{b})] 
 $N_2=\begin{pmatrix}
1&x_1&x_2&x_3\\
0&1&0&-x_2\\
0& 0&1& x_1\\
0&0&0&1
\end{pmatrix}$, if $k=1$, $d(U)=2$ and $\char \F\neq 2$;
\item[(\rm{c})] 
 $N_{3,\lambda}=\begin{pmatrix}
1&x_1&x_2&x_3\\
0&1&0&x_1+ x_2\\
0&0&1& \lambda x_2\\
0&0&0&1
\end{pmatrix}$, $\lambda \in \F^\ast$, if $k=1$ and $d(U)=3$.
\end{itemize}
 The algebras $\A_1=\F I_4 +N_1$, $\A_2=\F I_4 +N_2$ and $\A_{3,\lambda}=\F I_4 +N_{3,\lambda}$
 have the following 
presentation:
$$
\begin{array}{lcl}
\A_1=\sp_{\F}(t_1,t_2),\quad &\textrm{ where } & {t_1}^2={t_2}^2=t_1t_2=0;\\
\A_2=\sp_{\F}(t_1,t_2),\quad & \textrm{ where }  & {t_1}^2={t_2}^2=t_1t_2+t_2t_1=0;\\
\A_{3,\lambda}=\sp_{\F}(t_1,t_2),\quad & \textrm{ where } &  {t_2}^2-\lambda{t_1}^2=t_2t_1={t_1}^2-t_1t_2=0,\quad \lambda 
\in \F^\ast.
\end{array}
$$
\end{lemma}

\begin{proof}
If $U$ is abelian, then $r(U)=r(\Z(U))=1$ and by Lemma \ref{Tn}, $U=S_{(2,1^2)}$.
So, suppose that $U$ is not abelian. We may assume $z=\blockdiag(J_2,J_1,J_1)\in \Z(U)$. Now $z=\mu(v_1)$ 
gives $\delta(v_1)=0$ and, from \eqref{vivj},  we get that the first row of $\delta(v)$ is zero for all $v\in\F^3$. 
Hence
$\delta(v)=\left(\begin{smallmatrix} 0 & 0 & 0 \\ \alpha &  y_1 & y_2 \\ \beta & y_3 & -y_1                             
\end{smallmatrix}\right)$ with $y_1^2+y_2y_3=0$, by the unipotency of $U$. 
Now fix $v=(0,x_2,x_3)\neq 0$ in $\langle v_2,v_3\rangle$. Conjugating by $g=\diag(I_2,P)$ with a suitable $P\in 
\GL_2(\F)$, we may assume either (\emph{i}) $y_1=y_2=0$ and $y_3=1$ or (\emph{ii}) $y_1=y_2=y_3=0$. 

In case (\emph{i}) if $x_3=0$, we may suppose $v=v_2$. So $\delta(v_2)=\left(\begin{smallmatrix} 0 & 0 & 0 \\ \alpha_2 
& 0 & 0 \\ \beta_2 & 1 & 0    \end{smallmatrix}\right)$ and $\delta(v_3)=\left(\begin{smallmatrix} 0 & 0 & 0 \\ 
\alpha_3 
& \gamma_1 & \gamma_2 \\ \beta_3 & \gamma_3 & -\gamma_1   \end{smallmatrix}\right)$.
Applying \eqref{prodotto} to $v_2,v_2$ and to $v_3,v_2$ we obtain respectively
$\alpha_2=0$ and $\gamma_1=-1$, $\gamma_2=0$, which contradicts the unipotency of $U$.
On the other hand, if $x_3\neq 0$, conjugating by  
$\diag\left(I_2,\left(\begin{smallmatrix} x_3^{-1} & 0 \\ -x_2x_3^{-2} & x_3^{-1}\end{smallmatrix}\right)\right)$ we 
may 
assume $v=v_3$, hence
$\delta(v_3)=\left(\begin{smallmatrix} 0 & 0 & 0 \\ \alpha_3 & 0 & 0 \\ \beta_3 & 1 & 0  
\end{smallmatrix}\right)$.
Applying \eqref{prodotto} to $v_3,v_3$ we obtain $\delta(v_2)=\delta(v_3)^2=\left(\begin{smallmatrix} 0 & 0 & 0 \\ 0 & 
0 
& 0 \\ \alpha_3& 0 & 0                                                                                                  
 \end{smallmatrix}\right)$. However, in this case $U$ is abelian.

In case (\emph{ii}), up to a further conjugation by a matrix of the same shape of $g$, we may suppose $v=v_2$, hence
$\delta(v_2)=\left(\begin{smallmatrix} 0 & 0 & 0 \\ \alpha_2 & 0 & 0 \\ \beta_2 & 0 & 0    \end{smallmatrix}\right)$.
Set $\delta(v_3)=\left(\begin{smallmatrix} 0 & 0 & 0 \\ \alpha_3 & \gamma_1 & \gamma_2 \\ \beta_3 & \gamma_3 & 
-\gamma_1 
\end{smallmatrix}\right)$. Now,  \eqref{prodotto} applied to $v_2,v_3$ 
gives $\gamma_1\delta(v_2)+\gamma_2\delta(v_3)=\delta(v_2)\delta(v_3)=0$. In particular, $\gamma_2^2=0$, whence 
$\gamma_2=0$.
It follows $\gamma_1=0$ by the condition $\gamma_1^2+\gamma_2\gamma_3=0$.  
Replacing, if necessary, $v_3$ by a scalar multiple, we get either $\gamma_3=1$ or $\gamma_3=0$.
In the first case, \eqref{prodotto} applied to $v_3,v_3$ gives $\alpha_3 =\beta_2$, $\alpha_2=0$  and
$U$ is abelian. In the second one, $U$ is conjugate 
$V(\alpha_2,\alpha_3,\beta_2,\beta_3)$  
with $\alpha_3\neq \beta_2$, from the non-abelianity.

Let $\Delta=\alpha_2\beta_3-\alpha_3\beta_2$, with $\alpha_3\neq \beta_2$.
If $\Delta=0$, then $k(U)=2$.
An isomorphism $\Psi: \A_1\to \A=\F I_4+V(\alpha_2,\alpha_3,\beta_2,\beta_3)$ is obtained by setting
$$
\begin{array}{rl}
\textrm{if } \alpha_2\neq 0 \textrm{ take} &
\left\{\begin{array}{l}
\Psi(t_1)=\beta_2 X_2-\alpha_2 X_3\\
\Psi(t_2)=\alpha_3 X_2-\alpha_2 X_3
\end{array}\right. ;\\[10pt]
\textrm{if } \alpha_2=\beta_2=0 \textrm{ take} &
\left\{\begin{array}{l}
\Psi(t_1)=-\frac{\beta_3}{\alpha_3}  X_2+ X_3\\
\Psi(t_2)=X_2
\end{array}\right. ;\\[10pt]
\textrm{if } \alpha_2=\alpha_3=0 \textrm{ take} &
\left\{\begin{array}{l}
\Psi(t_1)=X_2\\
\Psi(t_2)=-\frac{\beta_3}{\beta_2}  X_2+ X_3
\end{array}\right. .
\end{array}
$$
By  Proposition \ref{prop:algebre}, the subgroup $U$ is conjugate to $N_1$. 

Suppose now that $\Delta\neq 0$ (which implies $k(U)=1$).
If $d(U)=2$, then $\char \F\neq 2$, $\alpha_2=\beta_3=0$ and $\beta_2=-\alpha_3\neq 0$. 
An isomorphism $\Psi: \A_2\to \A=\F I_4+V(\alpha_2,\alpha_3,\beta_2,\beta_3)$ is obtained by 
setting 
$$\Psi(t_1)=X_2,\quad \Psi(t_2)=X_3.$$ By  Proposition \ref{prop:algebre}, the subgroup $U$ is conjugate to 
$N_2$.

If $d(U)=3$, an isomorphism $\Psi: \A_{3,\lambda}\to \A=\F I_4+V(\alpha_2,\alpha_3,\beta_2,\beta_3)$ is obtained by setting $\lambda=\frac{\alpha_2\beta_3-\alpha_3\beta_2}{(\alpha_3-\beta_2)^2}$ and
$$
\begin{array}{rl}
\textrm{if } \beta_3\neq 0 \textrm{ take} &
\left\{\begin{array}{l}
\Psi(t_1)=X_3\\
\Psi(t_2)=-\frac{\beta_3}{\alpha_3-\beta_2} X_2+\frac{\alpha_3}{\alpha_3-\beta_2} X_3\\
\end{array}\right. ;\\[10pt]
\textrm{if } \beta_3=0, \alpha_2\neq 0 \textrm{ take} &
\left\{\begin{array}{l}
\Psi(t_1)=X_2\\
\Psi(t_2)=-\frac{\beta_2}{\alpha_3-\beta_2}X_2+\frac{\alpha_2}{\alpha_3-\beta_2}X_3\\
\end{array}\right. ;\\[10pt]
 \textrm{if } \alpha_2=\beta_3=0\; (\alpha_3\neq -\beta_2)
\textrm{ take} &
\left\{\begin{array}{l}
\Psi(t_1)=X_2+ \frac{\alpha_3-\beta_2}{\alpha_3} X_3\\
\Psi(t_2)=-\frac{\beta_2}{\alpha_3-\beta_2} X_2+ X_3\\
\end{array}\right. .
\end{array}
$$
By  Proposition \ref{prop:algebre}, the subgroup $V(\alpha_2,\alpha_3,\beta_2,\beta_3)$ is conjugate to 
$N_{3,\lambda}$, $\lambda\neq 0$. The statement now follows from Lemma \ref{cosquare}
\end{proof}

We can now classify the regular subgroups $U$ of $\AGL_3(\F)$ having linear $\delta$, including the non-abelian ones, arising from $d(\Z(U))=2$ and $r(\Z(U))=1$.

If $d(\Z(U))=4$, then $U$ is conjugate to $S_{(4)}$, by Lemma \ref{Cn}. 
If $d(\Z(U))=3$, then  $U$ is one of the subgroups 
$R_\alpha$ described in Lemma \ref{Cn-1}: namely, for  $k(U)=2$, $U$ is conjugate to $R_0$ 
and, for $k(U)=1$, $U$ is conjugate to $R_\lambda$, where $\lambda$ can be chosen in a transversal $\F^\square$ of $(\F^\ast)^2$ in $\F^\ast$. As observed in  Remark \ref{rem01}, $R_0$ 
and $R_1$ coincide, respectively, with  $S_{(3,1)}$ and $S_{(2,2)}^\s$. 
If $d(\Z(U))=2$ we have two possibilities. When $r(\Z(U))=2$ we apply  Lemma \ref{J2J2}, that gives
$\char \F=2$ and $U$ is conjugate to $U^3_1$. When $r(\Z(U))=1$ we apply Lemma \ref{J2:3}: when abelian
$U=S_{(2,1^2)}$, otherwise $U$ is conjugate to one of the subgroups $N_1,N_2, N_{3,\lambda}, \lambda \in \F^\ast$. 

A complete set of  representatives of the abelian regular subgroups 
of 
$\AGL_3(\F)$  is given in Table 
\ref{Tab5}. 
The conjugacy classes of non-abelian regular subgroups of $\AGL_3(\F)$ are described in Lemma 
\ref{J2:3}.

\begin{table}[!ht]
\begin{tabular}{cccccc}
$U$ & $\F I_{4} +  U$ & $\char \F$ & $\Psi$ & $\Ker(\Psi)$\\\hline 
\noalign{\smallskip}
$S_{(4)}$ & $\begin{pmatrix}
x_0&x_1&x_2&x_3\\
0&x_0&x_1&x_2\\
0&0&x_0&x_1\\
0&0&0&x_0
\end{pmatrix}$ & any & \eqref{Psilam}  & $\langle {t_1}^4\rangle$ & indec.\\ 
\noalign{\smallskip}
 $S_{(3,1)}$ & $\begin{pmatrix}
x_0&x_1&x_2&x_3\\
0&x_0&x_1&0\\
0&0&x_0&0\\
0&0&0&x_0
\end{pmatrix}$ & any & \eqref{eq:U0} &  $\langle {t_1}^3, {t_2}^2,  t_1 t_2\rangle$ 
\\\noalign{\smallskip}
$\begin{array}{c}
R_\lambda, \;\lambda \in \F^\square\\[-5pt]
{}_{(R_1=S_{(2,2)}^\s)}
 \end{array}$
& $\begin{pmatrix}
x_0&x_1&x_2&x_3\\
0&x_0& 0 & x_1 \\
0&0&x_0& \lambda x_2\\
0&0& 0 &x_0
\end{pmatrix}$ & any &  \eqref{eq:Ualpha odd} & $\langle {t_1}^2 -\lambda {t_2}^2, t_1 t_2\rangle$ & 
indec.
\\\noalign{\smallskip}
 $U^3_{1}$ & $\begin{pmatrix}
x_0&x_1&x_2&x_3\\
0&x_0&0&x_2\\
0&0&x_0&x_1\\
0&0&0&x_0
\end{pmatrix}$ & $2$ & \eqref{eq:22} & 
$\langle {t_1}^2, {t_2}^2\rangle$ & indec.\\ 
\noalign{\smallskip}
 $S_{(2,1^2)}$ & $\begin{pmatrix}
x_0&x_1&x_2&x_3\\
0&x_0&0&0\\
0&0&x_0&0\\
0&0&0&x_0
\end{pmatrix}$ & any & \eqref{Psilam} & $\langle t_1, t_2, t_3\rangle ^2$ \\
\noalign{\smallskip}
 \end{tabular}
\caption{Representatives for the conjugacy classes of abelian regular subgroups $U$ of $\AGL_3(\F)$, for any field $\F$.}\label{Tab5}
\end{table}

\subsection{Case $n=4$} 

Once again, to obtain the full classification of the abelian regular subgroups of $\AGL_4(\F)$ we need some preliminary 
results. 

\begin{lemma}\label{J3J2} 
Let $U$ be a regular subgroup of $\AGL_4(\F)$ such that $\delta$ is linear.
If $d(\Z(U))=r(\Z(U))=3$, then $U$ is abelian and  is conjugate to
$$R(\alpha,\beta)=\begin{pmatrix}
1&x_1&x_2&x_3&x_4\\
0&1&x_1& 0 & x_3 \\
0&0&1& 0 & 0\\
0&0&\beta x_3 &1 & x_1+\alpha x_3\\
0&0&0&0&1
\end{pmatrix},\qquad \alpha,\beta\in \F.$$
Furthermore, if $\F$ has no quadratic extensions, there are exactly two conjugacy classes of such subgroups, whose 
representatives are, for instance, $R(0,0)$ and $R(1,0)$, which is conjugate to $S_{(3,2)}$.
Finally, $U_1^4=R(0,0)$ is indecomposable.
\end{lemma}

\begin{proof} 
We may suppose that $z=\blockdiag(J_3 ,J_2)\in \Z(U)$.
From $z,z^2\in U$ we obtain 
$\delta(v_1)=E_{1,2}+E_{3,4}$ and $\delta(v_2)=0$.
By Lemma \ref{centralizzatore} an the unipotency of $U$ it follows that 
$\delta(v_3)=E_{1,4}+\alpha E_{3,4}+\beta E_{3,2}+\gamma (E_{3,1}+E_{4,2})$.
Applying \eqref{prodotto} to $v_3,v_1$ we obtain $\delta(v_4)=\delta(v_3)\delta(v_1)=\gamma E_{3,2}$. 
In particular $U$ is abelian. Now $(\mu(v_3)-I_5)^3=0$ gives $\gamma=0$.
We conclude that $U$ is conjugate to $R(\alpha,\beta)$.

Now, assume that $\F$ has no quadratic extensions.
An epimorphism $\Psi: \F[t_1,t_2]\to \F I_{n+1} + R(\alpha,\beta)$ is obtained in the following way.
If $\alpha^2+4\beta\neq 0$,  take
\begin{equation}\label{eq:32-1}
\Psi(t_1) =  (\alpha+\sqrt{\alpha^2+4\beta}) X_1-2 X_3,\quad
\Psi(t_2)  =  (\alpha-\sqrt{\alpha^2+4\beta}) X_1- 2X_3 
\end{equation}
when $\char \F\neq 2$ and take
\begin{equation}\label{eq:32-2}
\Psi(t_1)  =  r X_1+X_3,\quad \Psi(t_2)  =  (r+\alpha) X_1 +X_3,
\end{equation}
when $\char \F=2$ (here $r\in \F$ is such that $r^2+\alpha r+\beta=0$).
In both cases, $\Ker(\Psi)=\langle {t_1}^3, {t_2}^3, t_1t_2 \rangle$. Comparison with the presentation of
$\F I_5+S_{(3,2)}$ given in \eqref{presUlam} shows that $R(\alpha,\beta)$ is conjugate to $S_{(3,2)}$
by Proposition \ref{prop:algebre}. 

If $\alpha^2+4\beta= 0$, take 
\begin{equation}\label{eq:32-3}
\Psi(t_1)  =   X_1,\quad \Psi(t_2)  =  \alpha X_1- 2X_3 
\end{equation}
when $\char\F\neq 2$ and take
\begin{equation}\label{eq:32-4}
\Psi(t_1)  =   X_1,\quad \Psi(t_2)  =  \sqrt{\beta} X_1 + X_3
\end{equation}
when $\char\F= 2$. In both cases $\Ker(\Psi)=\langle {t_1}^3, {t_1}^2 t_2, {t_2}^2\rangle$.

The algebras defined by the two presentations above are not isomorphic, as
the subspaces consisting of elements whose square is zero (namely $\langle {t_1}^2, {t_2}^2\rangle $ in the 
first case, $\langle {t_1}^2, t_2, t_1t_2\rangle $ in the second case) have different dimensions.
By Proposition \ref{prop:algebre}, there are exactly two conjugacy classes of subgroups $R(\alpha,\beta)$, 
depending on the nullity of $\alpha^2+4\beta$.

Noting that $k(U_1^4)=2$, direct computation shows that $U_1^4$ is indecomposable.
\end{proof}

\begin{lemma}\label{J3J1J1} 
Let $U$ be an abelian regular subgroup of $\AGL_4(\F)$.
If $d(U)=3$ and $r(U)=2$, then  
$U$ is conjugate to 
$$R(\alpha,\beta,\gamma)=\begin{pmatrix}
1& x_1& x_2& x_3& x_4\\
0 & 1 & x_1 & 0 & 0\\
0&0&1 & 0 & 0\\
0&0& \alpha x_3+\beta x_4 &1& 0\\
0&0& \beta x_3+\gamma x_4 & 0 &1\\
\end{pmatrix}.$$
If $\beta^2-\alpha\gamma\neq 0$, then $k(U)=1$.
If $\beta^2-\alpha\gamma=0$ with $(\alpha, \beta, \gamma)\neq (0,0,0)$, then $k(U)=2$ .
Finally, if $\alpha=\beta=\gamma=0$, then $k(U)=3$.
Furthermore, assuming that every  element of $\F$ is a square,  there are exactly 
three conjugacy classes of such subgroups. Their representatives are, for instance, 
$$U_2^4=R(0,0,1)\;\; (\textrm{if } k=1),\quad R(0,1,0) \;\; (\textrm{if } 
k=2),\quad R(0,0,0)=S_{(3,1,1)}\;\; (\textrm{if } k=3).$$
Observe that $R(0,1,0)$ is conjugate to $S_{(2,2,1)}^\s$.
\end{lemma}

\begin{proof} 
We may assume $z=\blockdiag(J_3 ,J_1, J_1)\in U$.
From $z,z^2\in U$ we obtain 
$\delta(v_1)=E_{1,2}$ and $\delta(v_2)=0$.
By Lemma \ref{centralizzatore} and the unipotency of $U$, for any $v \in \langle v_3,v_4\rangle$ 
we have $\delta(v)=\left(\begin{smallmatrix}
0 & 0 & 0 & 0 \\ 0 & 0 & 0 & 0 \\ 0 & \alpha & \xi & \eta \\
0 & \beta & \epsilon & -\xi \end{smallmatrix}\right)$, where $\xi^2+\eta\epsilon=0$.
Now, $\rk(\mu(v_1+v_3+v_4)-I_5)=3$ implies
$\delta(v_3)=\alpha_3 E_{3,2}+\beta_3 E_{4,2}$ 
and $\delta(v_4)=\alpha_4 E_{3,2}+\beta_4 E_{4,2}$.
Since $U$ is abelian, \eqref{vivj} applied to $v_3,v_4$ implies 
$\alpha_4=\beta_3$. Hence $U$ is conjugate to $R(\alpha,\beta,\gamma)$.

Suppose first  that $\Delta=\beta^2-\alpha\gamma\neq 0$. In this case $k(U)=1$.
If $\char \F \neq 2$ define
\begin{equation}\label{eq:k1-3}
\begin{array}{ll}
\textrm{if }\alpha\neq 0 & \textrm{if }\alpha =0\\\hline\noalign{\smallskip}
\begin{array}{rcl}
\Psi(t_1) & = & X_1,\\
\Psi(t_2) & = & \frac{\beta+\sqrt{\Delta}}{2\Delta} X_3 - 
\frac{\alpha}{2\Delta}X_4,\\[2pt]
\Psi(t_3) & = &\frac{-\beta+\sqrt{\Delta}}{\alpha} X_3+X_4,
\end{array}& 
\begin{array}{rcl}
\Psi(t_1) & = & \beta X_1,\\
\Psi(t_2) & = &  -\frac{\gamma}{2} X_3+ \beta X_4,\\
\Psi(t_3) & = & X_3,
\end{array}
 \end{array}
 \end{equation}
and if $\char \F=2$, define
\begin{equation}\label{eq:k1-2}
 \begin{array}{rcl}
\Psi(t_1) & = & \sqrt[4]{\Delta} X_1+\frac{\sqrt{\gamma}}{\sqrt[4]{\Delta}} 
X_3+\frac{\sqrt{\alpha}}{\sqrt[4]{\Delta}}X_4,
\\
\Psi(t_2) & = & \sqrt{\alpha} X_1+X_3,\\
\Psi(t_3) & = &  \sqrt{\gamma} X_1+X_4.
\end{array}
\end{equation}
We get $\Ker(\Psi)=\langle {t_1}^2-t_2t_3, {t_2}^2, {t_3}^2,t_1t_2, t_1t_3\rangle$.

Next suppose $\beta^2-\alpha\gamma=0$ and $(\alpha,\beta,\gamma)\neq (0,0,0)$.
Then $k(U)=2$. Define
\begin{equation}\label{eq:k2-1}
\begin{array}{ll}
\textrm{ if }\alpha\neq 0 & \textrm{ if }\alpha= 0 \textrm{ and } \gamma\neq 0\\\hline\noalign{\smallskip}
\begin{array}{rcl}
\Psi(t_1) & = & X_1,\\
\Psi(t_2) & = & -\frac{\beta\sqrt{\alpha}+1}{\sqrt{\alpha}} X_3 +\alpha X_4,\\
\Psi(t_3) & = & -\beta X_3+\alpha X_4,
\end{array} &  
\begin{array}{rlc}
\Psi(t_1) & = & \sqrt{\gamma}X_1,\\
\Psi(t_2) & = &  X_4,\\
\Psi(t_3) & = &  X_3.
\end{array}
\end{array}
\end{equation}
In both cases,  $\Ker(\Psi)=\langle {t_1}^2- {t_2}^2, {t_3}^2, t_1t_2, t_1t_3, t_2t_3 \rangle$.
Comparison with the presentation of $\F I_5+S_{(2,2,1)}^\s$ given in \eqref{presUmu} shows
that $U$ is conjugate to $S_{(2,2,1)}^\s$.

Finally, if $\alpha=\beta=\gamma=0$, then $R(0,0,0)=S_{(3,1,1)}$ and $k(U)=3$.
\end{proof}

\begin{rem}\label{remPoo}
Consider the algebra $\A$ of Example \ref{Poon} and the corresponding regular subgroup $R=R(-1,0,-1)$ 
of the previous Lemma. Since $k(R)=1$, under the assumption  that $-1$ is a square in $\F$, we obtain 
that the algebras
$$\frac{\F[t_1,t_2,t_3]}{\langle {t_1}^2+{t_2}^2, {t_1}^2+{t_3}^2, t_1 t_2, t_1t_3, t_2t_3\rangle }
\quad \textrm{ and } \quad
\frac{\F[t_1,t_2,t_3]}{\langle {t_1}^2-t_2t_3, {t_2}^2, {t_3}^2,t_1t_2, t_1t_3\rangle}$$
are isomorphic in any characteristic.  Actually, if $\char \F=2$,
an isomorphism can also be obtained directly via the change of variables $t_1'=t_1+t_2+t_3$, 
$t_2'=t_2+t_3$, $t_3'=t_1+t_3$. 
This fixes an inaccuracy of \cite{P}, corrected in \cite{P2}.
\end{rem}

\begin{lemma}\label{J2J2J1} 
Let $U$ be an abelian regular subgroup of $\AGL_4(\F)$.
If $d(U)=r(U)=2$, then $\char \F=2$ and $U$ is conjugate to  $U_1^3\times S_{(1)}$, where
$U_1^3$ is defined in Table $\ref{Tab5}$. 
\end{lemma}
 
\begin{proof}
We may suppose $z=\diag(J_2,J_2,J_1)\in U$. From $z\in U$ we get $\delta(v_1)=E_{2,3}$.
Lemma \ref{centralizzatore} and the condition $(\mu(v_2)-I_5)^2= 0$ give 
$\delta(v_2)=E_{1,3}+\alpha_2 E_{4,3}$. Now, $(\mu(v_1+v_2)-I_5)^2= 0$ implies $\char \F=2$.
Applying \eqref{prodotto} to $v_2,v_1$ we get $\delta(v_3)=\delta(v_2)\delta(v_1)=0$.
From  $(\mu(v_2+v_4)-I_5)^2= 0$ we obtain $\delta(v_4)=(\alpha_2+\alpha_4)E_{2,3}+\alpha_4 E_{4,3}$, but 
$(\mu(v_4)-I_5)^2= 0$ gives  $\alpha_4=0$.
In conclusion, $U$ is conjugate to 
$$R(\alpha)=\begin{pmatrix}
1&x_1&x_2&x_3&x_4\\
0&1&0& x_2 & 0\\
0&0&1 &x_1+\alpha x_4&0\\
0&0&0&1&0\\
0&0&0&\alpha x_2&1
\end{pmatrix}.$$

An epimorphism $\Psi: \F[t_1,t_2,t_3]\to \F I_{n+1} + R(\alpha)$ is obtained by setting 
\begin{equation}\label{eq:221}
\Psi(t_1)=X_1, \quad \Psi(t_2)=X_2, \quad \Psi(t_3)=\alpha X_1+X_4.
\end{equation}
We have $\Ker(\Psi)=\langle {t_1}^2, {t_2}^2, {t_3}^2, t_1t_3, t_2 t_3\rangle$. Considering the presentation
of $U_1^3$ given in Table $\ref{Tab5}$, we have that $U$ is conjugate to $U_1^3\times S_{(1)}$.
\end{proof}

We now classify the abelian regular subgroups $U$ of $\AGL_4(\F)$.
If $d(U)=5$, then $U$ is conjugate to $S_{(5)}$ by Lemma \ref{Cn}. 
If $d(U)=4$, then by,  Lemma \ref{Cn-1}, $U$ is conjugate to $R_0=S_{(4,1)}$ when  $k(U)=2$,  and  to $R_1=S_{(3,2)}^\s$ when $k(U)=1$.
Suppose  $d(U)=3$. If $r(U)=3$, then $U$ is conjugate  either to $U_1^4$ or to $S_{(3,2)}$ by Lemma \ref{J3J2}.
If $r(U)=2$, then $U$ is conjugate either to $U_2^4$ or to $S_{(2,2,1)}^\s$ or to $S_{(3,1,1)}$ by Lemma \ref{J3J1J1}.
Finally suppose  $d(U)=2$. If $r(U)=1$ then $U$ is conjugate to $S_{(2,1^3)}$ by Lemma \ref{Tn}.
If $r(U)=2$, then $\char\F=2$ and $U$ is conjugate to $U_{1}^3\times S_{(1)}$ by Lemma \ref{J2J2J1}.

When $\F$ has no quadratic  extensions, a  complete set of representatives of the conjugacy classes of abelian regular subgroups is given in Tables \ref{Tab6} and \ref{Tab7}.

\begin{table}[!ht]
\begin{tabular}{cccccc}
$U$ & $\F I_{5}+  U$ & $\char\F$ & $\Psi$ & $\Ker(\Psi)$.
\\\hline
\noalign{\smallskip}
$S_{(5)}$ & $\begin{pmatrix}
x_0&x_1&x_2&x_3&x_4\\
0&x_0&x_1&x_2&x_3\\
0&0&x_0&x_1&x_2\\
0&0&0&x_0&x_1\\
0&0&0&0&x_0
\end{pmatrix}$ & any & \eqref{Psilam} & $\langle {t_1}^5\rangle$ \\\noalign{\smallskip}
$S_{(3,2)}^\s$ & $\begin{pmatrix}
x_0&x_1&x_2&x_3&x_4\\
0&x_0&x_1&0&x_2\\
0&0&x_0&0&x_1\\
0&0&0&x_0&x_3\\
0 & 0  & 0 & 0 &x_0
\end{pmatrix}$ & any & \eqref{Psimu} & $\langle {t_1}^3-{t_2}^2, t_1 t_2\rangle $ \\\noalign{\smallskip}
$U^4_{1}$ & 
$\begin{pmatrix}
x_0&x_1&x_2 & x_3& x_4\\
0&x_0&x_1& 0 & x_3 \\
0&0&x_0&0 &0\\
0&0&0& x_0 &x_1\\
0&0&0&0&x_0
\end{pmatrix}$ & 
$\begin{array}{c}
 \neq 2\\\\ 2
\end{array}$
 &  
$ \begin{array}{c}
\eqref{eq:32-1}\\ \\ \eqref{eq:32-2} 
 \end{array}$
 & $\langle {t_1}^3, {t_1}^2 t_2, {t_2}^2\rangle$ \\
\noalign{\smallskip}
$U^4_{2}$ & $\begin{pmatrix}
x_0&x_1&x_2&x_3&x_4\\
0&x_0& 0&0&x_1\\
0&0&x_0&0&x_3\\
0&0& 0 & x_0&x_2\\
0&0& 0 &0&x_0\\
\end{pmatrix}$ & 
$\begin{array}{c}
 \neq 2\\\\2
\end{array}$
 &  $\begin{array}{c}
 \eqref{eq:k1-3}\\\\  \eqref{eq:k1-2}
  \end{array}$
 & $\langle {t_1}^2-t_2t_3, {t_2}^2, {t_3}^2,   t_1t_2, t_1t_3\rangle$\\
\noalign{\smallskip}
\end{tabular}
\caption{Representatives for the conjugacy classes of indecomposable abelian regular subgroups $U$ of $\AGL_4(\F)$, when $\F$ has no quadratic extensions.}\label{Tab6}
\end{table}

\begin{table}[!ht]
\begin{tabular}{cccccc}
$U$ & $\F I_{5} + U$ & $\char \F$ & $\Psi$ &  $\Ker(\Psi)$
\\\hline
\noalign{\smallskip}
$S_{(4,1)}$ & $\begin{pmatrix}
x_0&x_1&x_2&x_3&x_4\\
0&x_0&x_1&x_2&0\\
0&0&x_0&x_1&0\\
0&0&0&x_0&0\\
0&0&0&0&x_0
\end{pmatrix}$ & any & \eqref{Psilam} & $\langle {t_1}^4,{t_2}^2,t_1t_2 \rangle$ \\\noalign{\smallskip}
$S_{(3,2)}$ & 
$\begin{pmatrix}
x_0&x_1&x_2 & x_3& x_4\\
0&x_0&x_1& 0 &0 \\
0&0&x_0&0 &0\\
0&0& 0& x_0 & x_3\\
0&0&0&0&x_0
\end{pmatrix}$ & any & \eqref{Psimu} & $\langle {t_1}^3,{t_2}^3, t_1t_2 \rangle$  \\\noalign{\smallskip}
$S_{(3,1,1)}$  & $\begin{pmatrix}
x_0&x_1&x_2&x_3&x_4\\
0&x_0& x_1&0&0\\
0&0&x_0& 0&0\\
0&0&0& x_0&0\\
0&0&0&0 &x_0\\
\end{pmatrix}$ & any  & \eqref{Psilam} &
$\langle {t_1}^3,  {t_2}^2, {t_3}^2,  t_1t_2,  t_1t_3, t_2t_3\rangle$  \\\noalign{\smallskip}
 $S_{(2,2,1)}^\s$ & $\begin{pmatrix}
x_0&x_1&x_2&x_3&x_4\\
0&x_0& 0&  x_1&0\\
0&0&x_0&x_2&0\\
0&0& 0 & x_0&0\\
0&0& 0 &0&x_0\\
\end{pmatrix}$ &any & \eqref{Psimu} &
$\langle {t_1}^2-{t_2}^2, {t_3}^2,  t_1t_2, t_1t_3, t_2t_3\rangle$  \\\noalign{\smallskip}
$U_1^3\times S_{(1)}$ & $\begin{pmatrix}
x_0& x_1 & x_2 & x_3& x_4\\
0&x_0&0&x_2& 0\\
0&0&x_0&x_1&0\\
0&0&0&x_0&0\\
0&0&0&0&x_0\\
\end{pmatrix}$ & $2$ & \eqref{eq:221} & 
$\langle  {t_1}^2, {t_2}^2, {t_3}^2, t_1t_3, t_2 t_3 \rangle$ \\ \noalign{\smallskip}
$S_{(2,1^3)}$  & $\begin{pmatrix}
x_0&x_1&x_2&x_3&x_4\\
0&x_0&0&0&0\\
0&0&x_0&0&0\\
0&0&0&x_0&0\\
0&0&0&0&x_0
\end{pmatrix}$ & any & \eqref{Psilam} & $\langle t_1, t_2, t_3,t_4\rangle ^2$ \\
\noalign{\smallskip}
\end{tabular}
\caption{Representatives for the conjugacy  classes of decomposable abelian regular subgroups $U$ of $\AGL_4(\F)$, when $\F$ has no quadratic extension.}\label{Tab7}
\end{table}

\subsection*{Acknowledgments} 
We are grateful to Marco Degiovanni for a very useful suggestion, which led to  Example 
\ref{degio}.


\begin{thebibliography}{15}

\bibitem{CDS} A. Caranti, F. Dalla Volta \and M. Sala, Abelian regular subgroups of the affine group and radical rings,
\emph{Publ. Math. Debrecen} \textbf{69} (2006), no. 3,  297--308.

\bibitem{CCS} F. Catino, I. Colazzo \and P. Stefanelli, On regular subgroups of the affine group, \emph{Bull. Aust. Math. Soc.}
\textbf{91} (2015), no. 1, 76--85.

\bibitem{CCS2} F. Catino, I. Colazzo \and P. Stefanelli, Regular subgroups of the affine group and asymmetric 
product of radical braces, preprint.


\bibitem{CR} F. Catino \and R. Rizzo, Regular subgroups of affine group and radical circle algebras,
\emph{Bull. Aust. Math. Soc.} \textbf{79} (2009), 103--107.

\bibitem{Ch} L.N. Childs, Elementary abelian Hopf Galois structures and polynomial formal groups, \emph{J. Algebra} 
\textbf{283} (2005), 292--316.

\bibitem{Ch2} L.N. Childs, On abelian Hopf Galois structures and finite commutative nilpotent rings,
\emph{New York J. Math.} \textbf{21} (2015), 205--229. 

\bibitem{DG} W. De Graaf, Classification of nilpotent associative algebras of small dimension, \texttt{arXiv:1009.5339}, 27 Sep 2010.

 \bibitem{H2000} P. Heged\H{u}s, Regular subgroups of the affine group, \emph{J. Algebra} \textbf{225} (2000), no. 2, 740--742.

\bibitem{HS} R.A. Horn \and V.V. Sergeichuk, Canonical forms for complex matrix congruence and ${}^\ast$congruence, \emph{Linear Algebra Appl.} \textbf{416} (2006), no. 2-3, 1010--1032. 

 
\bibitem{H} J.E. Humphreys, \emph{Linear algebraic groups}, Graduate Texts in Mathematics, No. 21. 
Springer-Verlag, New York-Heidelberg, 1975.


\bibitem{J} N. Jacobson, \emph{Basic Algebra II},  W.H. Freeman and Company, New York, (1989). 

\bibitem{P} B. Poonen, Isomorphism types of commutative algebras of finite rank over an algebraically closed field, in Computational arithmetic geometry, 111--120, \emph{Contemp. Math.}, \textbf{463}, Amer. Math. Soc., Providence, RI, 2008. 

\bibitem{P2} \texttt{http://www-math.mit.edu/$\sim$poonen/papers/dimension6.pdf}, 11 January 2016.
\bibitem{T} M.C. Tamburini Bellani,
Some remarks on regular subgroups of the affine group,
\emph{Int. J. Group Theory} \textbf{1} (2012), no. 1, 17--23.
 \end{thebibliography}
\end{document}